\documentclass{amsart}
\usepackage{hyperref}
\usepackage{amsmath}
\usepackage{amsthm}
\usepackage{graphicx,color}
\usepackage{float}
\usepackage[margin=1in]{geometry}

\usepackage{url}
\usepackage{adjustbox}
\usepackage{thm-restate}

\usepackage{tikz, braids}
\usetikzlibrary{calc,decorations.markings,positioning}
\usetikzlibrary{shapes.arrows}
\pgfdeclarelayer{bg}    
\pgfsetlayers{bg,main}  

\newcounter{braid}
\newcounter{strands}
\pgfkeyssetvalue{/tikz/braid height}{1cm}
\pgfkeyssetvalue{/tikz/braid width}{1cm}
\pgfkeyssetvalue{/tikz/braid start}{(0,0)}
\pgfkeyssetvalue{/tikz/braid colour}{black}
\pgfkeys{/tikz/strands/.code={\setcounter{strands}{#1}}}

\makeatletter
\def\cross{%
  \@ifnextchar^{\message{Got sup}\cross@sup}{\cross@sub}}

\def\cross@sup^#1_#2{\render@cross{#2}{#1}}

\def\cross@sub_#1{\@ifnextchar^{\cross@@sub{#1}}{\render@cross{#1}{1}}}

\def\cross@@sub#1^#2{\render@cross{#1}{#2}}

\def\render@cross#1#2{
  \def\strand{#1}
  \def\crossing{#2}
  \pgfmathsetmacro{\cross@y}{-\value{braid}*\braid@h}
  \pgfmathtruncatemacro{\nextstrand}{#1+1}
  \foreach \thread in {1,...,\value{strands}}
  {
    \pgfmathsetmacro{\strand@x}{\thread * \braid@w}
    \ifnum\thread=\strand
    \pgfmathsetmacro{\over@x}{\strand * \braid@w + .5*(1 - \crossing) * \braid@w}
    \pgfmathsetmacro{\under@x}{\strand * \braid@w + .5*(1 + \crossing) * \braid@w}
    \draw[braid] \pgfkeysvalueof{/tikz/braid start} +(\over@x pt,\cross@y pt) to[out=-90,in=90] +(\under@x pt,\cross@y pt -\braid@h);
    \draw[braid] \pgfkeysvalueof{/tikz/braid start} +(\under@x pt,\cross@y pt) to[out=-90,in=90] +(\over@x pt,\cross@y pt -\braid@h);
    \else
    \ifnum\thread=\nextstrand
    \else
     \draw[braid] \pgfkeysvalueof{/tikz/braid start} ++(\strand@x pt,\cross@y pt) -- ++(0,-\braid@h);
    \fi
   \fi
  }
  \stepcounter{braid}
}

\tikzset{braid/.style={double=\pgfkeysvalueof{/tikz/braid colour},double distance=1pt,line width=2pt,white}}

\theoremstyle{plain}

\theoremstyle{definition}
\newtheorem{definition}{Definition}

\newtheorem{theorem}{Theorem}

\newtheorem{lemma}{Lemma}
\newtheorem{corollary}{Corollary}

\newcommand{\bd}{\partial}
\newcommand{\set}[1]{ \{ #1 \} }
\newcommand{\vect}[1]{ {\textbf #1 } }

\newcommand{\Prod}[2]{\Pi_{#2}^{#1}}

\newcommand{\ProdPow}[3]{ \Big( \Prod{#1}{#2} \Big)^{#3} }

\newcommand{\BigProd}[4]{\Prod{#1}{#2} \cdots \Prod{#3}{#4} }

\newcommand{\as}{\alpha_*}
\newcommand{\bs}{\beta_*}
\newcommand{\gs}{\gamma_*}
\newcommand{\ds}{\delta_*}
\newcommand{\es}{\varepsilon_*}

\begin{document}

\markboth{E. Amoranto, B. Doleshal, and M. Rathbun}
{Additional Cases of Positive Twisted Torus Knots}



\title{Additional Cases of Positive Twisted Torus Knots}

\author{Evan Amoranto}
\address{Evan Amoranto; California State University, Fullerton} 
\email{eamoranto@csu.fullerton.edu}

\author{Brandy Doleshal}
\address{Brandy Doleshal; Sam Houston State University}
\email{bdoleshal@shsu.edu}

\author{Matt Rathbun}
\address{Matt Rathbun; California State University, Fullerton}
\email{mrathbun@fullerton.edu}

\maketitle

\begin{abstract}
A twisted torus knot is a knot obtained from a torus knot by twisting adjacent strands by full twists. The twisted torus knots lie in $F$, the genus 2 Heegaard surface for $S^3$. Primitive/primitive and primitive/Seifert knots lie in $F$ in a particular way. Dean gives sufficient conditions for the parameters of the twisted torus knots to ensure they are primitive/primitive or primitive/Seifert. Using Dean's conditions, Doleshal shows that there are infinitely many twisted torus knots that are fibered and that there are twisted torus knots with distinct primitive/Seifert representatives with the same slope in $F$. In this paper, we extend Doleshal's results to show there is a four parameter family of positive twisted torus knots. Additionally, we provide new examples of twisted torus knots with distinct representatives with the same surface slope in $F$.
\end{abstract}

\section{Introduction}	
A twisted torus knot is a knot obtained from a torus knot $T(p,q)$ by twisting $r$ adjacent strands $n$ full twists. We denote the twisted torus knot with these parameters by $K(p,q,r,n)$. Much work has been done recently in the pursuit of understanding the twisted torus knots. Lee classifies the twisted torus knots that are unknotted \cite{LeeTTKU} and gives an infinite family of twisted torus knots that are cable knots \cite{LeeTTKCK}. Morimoto describes an infinite family of composite twisted torus knots \cite{MorCTTK} and shows that for any positive integer $n$, there are infinitely many twisted torus knots with an $n$-string essential tangle decomposition \cite{MorTDTTK}. Moriah and Sedgwick provide an infinite class of hyperbolic twisted torus knots with a unique minimal genus Heegaard splitting \cite{MorHSTTK}. Bowman, Taylor and Zupan compute bridge numbers for a family of hyperbolic twisted torus knots and show that each knot in this family has two arbitrarily large gaps in its bridge spectrum \cite{BowBSTTK}. Doleshal \cite{DolFPSTTK} provides a class of twisted torus knots that are fibered by proving the following theorem.

\begin{theorem}[Theorem 3.1 of \cite{DolFPSTTK}]
\label{theorem:OriginalDoleshal}
The twisted torus knots $K(p, q, r, -n)$, with $n > 0$ and $r < q$, are positive when $nq < p$.
\end{theorem}

Note that by \cite{DeaHKSSFDS}, if $r < \min(p, q)$, then the knots $K(p, q, r, -n)$ and $K(q, p, r, -n)$ are isotopic, so this result actually holds for $K(p, q, r, -n)$ with $n > 0$ and $r < p$ when $np < q$ as well. We extend Theorem 3.1 of \cite{DolFPSTTK} with the following:

\begin{restatable}{theorem}{MainTheorem}\label{theorem:MainTheorem}
Suppose $p$, $q$, $r$, $n \in \mathbb N$ with $p$, $q$, $r \ge 2$ and $(p, q) = 1$. Further suppose that $k$ and $e$ are the integers with $k \ge 0$, $1 \le e <q $, and $p = kq + e$. The twisted torus knot $K(p,q,r, -n)$ is positive if $n < k + 1$, or if $n = k + 1$ and $r \le e$.
\end{restatable}

This theorem provides an infinite family of new examples beyond those discovered in \cite{DolFPSTTK}. In particular, we call attention to the condition that $ n \leq k+1$ in Theorem \ref{theorem:MainTheorem}. If $n \leq k$, then $nq \leq kq < p$, and the previous results apply, so the content of the extension is the case that $n = k+1$. If $n > k+1$, then there appears to be too much negative twisting introduced by the $n$ full twists, and we conjecture that the resulting twisted torus knots cannot be positive. In the case that $n = k+1$, however, we find a significant extension of known results, using $k$ full positive twists to cancel $n-1$ negative twists, and the remaining $e$ partial positive twists to cancel the remaining full negative twist.

The condition, then, that $e \geq r$ seems an important one. For instance, the twisted torus knot $K(4, 3, 2, -2)$, as noted in \cite{DolFPSTTK}, is not fibered, and thus does not have a positive braid representation. Note that in this case, $e < r$.

We offer two different proofs of Theorem 2, effectively describing the same simplification process, but viewed from fundamentally different perspectives. The first proof is algebraic, relying on calculations in the braid group, and is in the flavor of \cite{DolFPSTTK}. The second proof is geometric in nature, making use of surgery diagrams for a link. The two proofs can be seen to complement each other, either providing insight into an alternate interpretation of the other. Further, each viewpoint may appeal to the preferences of different readers, inviting both topologists and algebraists to consider questions like whether these hypotheses are sharp.

Once we establish the additional family of twisted torus knots that are positive, we will then be in a position to investigate the positions of these knots with respect to Heegaard surfaces in $S^3$, and introduce another infinite family of knots that have non-equivalent positions on the natural genus two Heegaard surface. Interestingly, the knots are all either torus knots or unknots.

Like the twisted torus knots, primitive/primitive and primitive/Seifert knots lie in the genus two Heegaard surface $F$ for $S^3$, where the primitive/Seifert knots are a generalization of the primitive/primitives knots. For compactness of notation, we sometimes use p/p to denote primitive/primitive and p/S to denote primitive/Seifert. The importance of the primitive/primitive and primitive/Seifert knots lies in the guarantee that they admit lens space and Seifert fibered space surgeries, respectively, at the surface slope of the knot with respect to $F$. Dean \cite{DeaSSFDSHK} gives a list of requirements on the parameters of a twisted torus knot that allows for the twisted torus knot to be primitive or Seifert with respect to one of the handlebodies bounded by $F$ and used these requirements to determine a class of twisted torus knots that are primitive/Seifert. Because it is not known that Dean's sufficient conditions to be a Seifert curve are also necessary conditions to be a Seifert curve, we refer to the curves that Dean describes as HEM-Seifert. (This term refers to Dean's hyper-, end- and middle-Seifert nomenclature.) 

Given a particular knot $K$, there can be two curves $\alpha$ and $\beta$ in $F$ so that $\alpha$ and $\beta$ are nonisotopic in $F$ but $\alpha$ and $\beta$ are both isotopic to $K$ in $S^3$. We call $\alpha$ and $\beta$ representatives of $K$. As shown in \cite{DolFPSTTK}, \cite{EudSFSDPSP}, and \cite{GunKDPPPSR}, there exist knots $K$ so that $\alpha$ and $\beta$ are both primitive/Seifert and have the same surface slope with respect to $F$, and there exist knots $K$ so that $\alpha$ is primitive/primitive and $\beta$ is primitive/Seifert while $\alpha$ and $\beta$ have the same surface slope with respect to $F$. In the former case, $K$ is called p/S-p/S and in the latter case, $K$ is called p/p-p/S.

Expanding on a theorem from \cite{DolFPSTTK}, we prove the following.

\begin{theorem}{DistinctPositionsOfKnots}\label{theorem:DistinctPositionsOfKnots}
For integers $k$, $q$, and $m$ with $k \in \set{0, 1}$, $q > 2$, $1 \leq m < q$, and $(q, m) = 1$, let $K_1$ and $K_2$ be the twisted torus knots $K_1 = K(kq + m, q, m, -1)$ and $K_2 = K(kq + q - m, q, q - m, -1)$, with their canonical embeddings on the genus two Heeagaard surface, $F$, for $S^3$. Then, $K_1$ and $K_2$ are isotopic as knots in $S^3$ and have the same surface slope with respect to $F$. 
\begin{enumerate}
\item[(i)]\label{Case1} When $k = 0$, and 
\begin{enumerate}
\item[(a)] $m = 1$, then $K_1$ is p/p, and $K_2$ is primitive with respect to one handlebody but neither primitive nor HEM-Seifert with resepct to the other.
\item[(b)] $m > 1$, then both $K_1$ and $K_2$ are primitive with respect to one handlebody, but neither primitive nor HEM-Seifert with respect to the other.
\end{enumerate}
\item[(ii)]\label{Case2} When $k = 1$, then both $K_1$ and $K_2$ are p/p.
\end{enumerate}
Further, when $k = 0$, there is no homeomorphism of $S^3$ sending the pair $(F, K_1)$ to $(F, K_2)$.
\end{theorem}

The case of $k \geq 2$ appears as a theorem in \cite{DolFPSTTK}. The proof that there is no homeomorphism sending $(F,K_1)$ to $(F, K_2)$ in that case employs the structure of the Seifert fibered space obtained from Dehn surgery on the knots at the shared surface slope with respect to $F$. In Case \ref{Case1}, the knots are not HEM-Seifert, so that method is not available to prove the corresponding result. Instead, we define the extended Goeritz group and use this new object to provide a homological obstruction to a such a homeomorphism.

In Section \ref{section:Definitions} we lay out important definitions. In Section \ref{section:Braids}, we present two proofs of Theorem \ref{theorem:MainTheorem}, one braid theoretic and one geometric. In Section \ref{section:InequivalentPositionsOfKnots}, we introduce the extended Goeritz group and prove Theorem \ref{theorem:DistinctPositionsOfKnots}.

\section{Definitions}
\label{section:Definitions}

\subsection{Positive Knots}

A crossing of an oriented knot diagram is said to be a \emph{positive} crossing if the orientation along the over-strand and then the under-strand obeys the right-hand rule. Otherwise, the crossing is said to be \emph{negative} (see Fig. \ref{fig:PositiveCrossing}). Note that reversing the orientation of the knot does not change whether each crossing is positive or negative, so the sign of a crossing is a property of the (unoriented) knot diagram.

A knot is said to be \emph{positive} if there exists a diagram for the knot with only positive crossings.

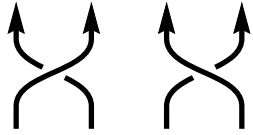
\begin{figure}[h]

\begin{adjustbox}{max totalsize={3in}{3in},center}
\begin{tikzpicture}
\braid[
 style all floors={fill=white},
 floor command={%
 \fill (\floorsx,\floorsy) rectangle (\floorex, \floorey);
 \draw (\floorsx,\floorsy) -- (\floorex,\floorsy);
 },
 line width=2pt,
 number of strands=4
 style strands={1}{},
] (braid) at (2,0) s_1^{-1}-s_3;
  \begin{scope}[every node/.style={draw,fill,single arrow,
    single arrow tip angle=30,
    single arrow head extend=2pt,
    single arrow head indent=1pt,
    inner sep=0pt}] 
  \node[shape border rotate=90] at (2,0) {};
  \node[shape border rotate=90] at (3,0) {};
  \node[shape border rotate=90] at (4,0) {};
   \node[shape border rotate=90] at (5,0) {};
    \end{scope}

\end{tikzpicture}

\end{adjustbox}

\caption{A positive crossing and a negative crossing, respectively, in a knot diagram.}
\label{fig:PositiveCrossing}

\end{figure}

\subsection{Twisted torus knots}

A torus knot is a knot that is isotopic to one lying on the surface of an unknotted torus in $S^3$. A torus knot, denoted as $T(p,q)$, can be parameterized by two co-prime integers, where $p$ denotes the number of times the knot wraps around the meridian of the torus and $q$ the number of times the knot wraps around the longitude. 

If we add additional full twists to some number of adjacent strands of a torus knot, the result is known as a \emph{twisted torus knot}. We define a twisted torus knot as a torus knot with an additional $n$ full twists on $r$ adjacent strands. Twisted torus knots can be parameterized by four numbers, $p, q, r$, and $n$, where $p,q$ indicate the original torus knot, $r$ is the number of strands to which additional twists are added, and $n$ is the number of full twists added to those strands, with the direction of twist determined by the sign of $n$. We denote a twisted torus knot by $K(p,q,r,n)$. Without loss of generality, a torus knot can always be represented with all positive crossings. So if we add a positive $n$ full twists on $r$ adjacent strands the resulting knot is still positive. There is a specific class of twisted torus knots that have been shown to be positive even if $n$ \emph{negative} full twists are added, namely when $r<q$ and $|n|q<p$ as shown by the second author.

\subsection{Braids}

The \emph{braid group on $m$ strands}, denoted $B_m$, is the group generated by $m-1$ elements, $\sigma_1, \sigma_2, \dots, \sigma_{m-1}$, with relations $\sigma_i \sigma_j = \sigma_j \sigma_i$ if $|i - j| > 1$, and $\sigma_i \sigma_{i+1} \sigma_{i} = \sigma_{i+1} \sigma_i \sigma_{i+1}$ for each $i = 1, \dots, m-2$. The latter family of relations is collectively referred to as the \emph{braid relation}. Thus,
\[ B_m = \langle \sigma_1, \sigma_2, \dots, \sigma_{m-1} \, | \, \sigma_i \sigma_j \sigma_i^{-1} \sigma_j^{-1},  i, j \in \set{1, \dots, m-1}, |i - j| > 1; \sigma_i \sigma_{i+1} \sigma_i \sigma_{i+1}^{-1} \sigma_i^{-1} \sigma_{i+1}^{-1}, i \in \set{1, \dots, m-2} \rangle. \] 

The group represents braiding among $m$ parallel strands of string, where the generator $\sigma_i$ corresponds to a half twist between the $i$th strand and $(i+1)$st strand in which the $(i+1)$st strand is pulled over the $i$th strand. The crossing created by the generator $\sigma_i$ is called a positive crossing. Similarly, the generator $\sigma_i ^{-1}$ creates a crossing between the $i$th strand and the $(i+1)$st strand except the $ith$ strand is pulled over the $(i+1)$st strand. The crossing created by the generator $\sigma_i ^{-1}$ is called a negative crossing. (See Fig. \ref{figure:Positive and Negative crossings}.)

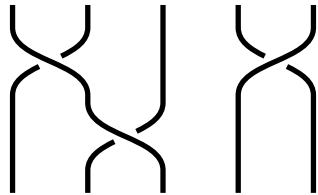
\begin{figure}[h]
\begin{center}
\begin{adjustbox}{max totalsize={3in}{3in},center}
\begin{tikzpicture}
\braid[
 style all floors={fill=white},
 floor command={%
 \fill (\floorsx,\floorsy) rectangle (\floorex, \floorey);
 \draw (\floorsx,\floorsy) -- (\floorex,\floorsy);
 },
 line width=2pt,
 number of strands=5
 style strands={1}{},
] (braid) at (2,0) s_1-s_4^{-1} s_2;

\end{tikzpicture}

\end{adjustbox}
\end{center}
\caption{Two negative crossings in the 3-strand braid $  \sigma_1^{-1} \sigma_2^{-1}$, and a positive crossing in a 2-strand braid $\sigma_1$, respectively.}
\label{figure:Positive and Negative crossings}
\end{figure}

The relations of the braid group are algebraic expressions of geometric relationships between braided configurations of string, in that one configuration can be deformed into the other without breaking any strands or disturbing the top or bottom endpoints. (See Figs. \ref{figure:Commutativity} and \ref{figure:BraidRelation}.)

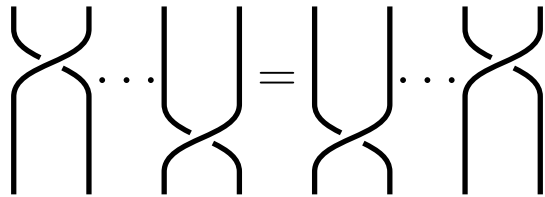
\begin{figure}[h]
\begin{center}
\begin{tikzpicture}
\braid[
 style all floors={fill=white},
 floor command={%
 \fill (\floorsx,\floorsy) rectangle (\floorex, \floorey);
 \draw (\floorsx,\floorsy) -- (\floorex,\floorsy);
 },
 line width=2pt,
 number of strands=8
] (braid)at (2,0) s_1^{-1}-s_7^{-1} s_3^{-1}-s_5^{-1} ;
\node[font=\Huge] at (5.5,-1) {\(=\)};
\node[font=\Huge] at (3.5, -1) {\(\cdots\)};
\node[font=\Huge] at (7.5, -1) {\(\cdots\)};
\end{tikzpicture}
\end{center}
\caption{Crossings between pairs of strands sufficiently distant from one another commute.}
\label{figure:Commutativity}
\end{figure}

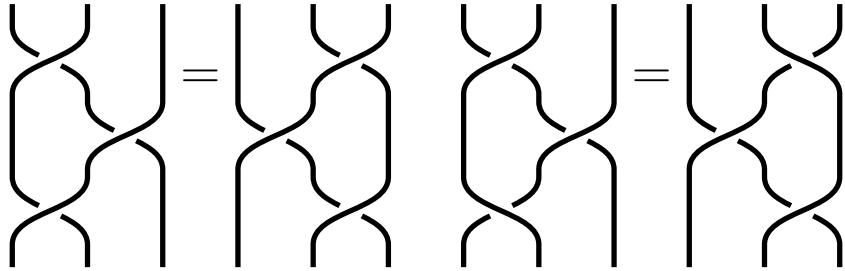
\begin{figure}[h]
\begin{center}
\begin{tikzpicture}
\braid[
 style all floors={fill=white},
 floor command={
 \fill (\floorsx,\floorsy) rectangle (\floorex, \floorey);
 \draw (\floorsx,\floorsy) -- (\floorex,\floorsy);
 },
 line width=2pt,
 number of strands=12
] (braid)at (2,0) s_1^{-1}-s_5^{-1}-s_7^{-1}-s_{11} s_2^{-1}-s_4^{-1}-s_8^{-1}-s_{10}^{-1} s_1^{-1}-s_5^{-1}-s_7-s_{11}^{-1};
\node[font=\Huge] at (4.5,-1) {\(=\)};
\node[font=\Huge] at (10.5,-1) {\(=\)};
\end{tikzpicture}
\end{center}
\caption{The braid relation(s).}
\label{figure:BraidRelation}
\end{figure}

A braid can be closed to form a knot or link by connecting the top and bottom of the $i$th strand for each $i = 1, 2, \dots, m$ in a way that does not introduce any additional crossings. This knot or link is called the \textit{closure} of the braid. If all of the generators in a braid word appear with positive exponents, then the braid is called a \emph{positive braid}. When the closure of a positive braid is a knot, the closure is a positive knot. Every knot can be expressed as the closure of a braid.

We will use a modified notation from that in \cite{DolFPSTTK}.

\begin{definition}
For $a \leq b$, let $\Pi_a^b = \sigma_b \cdot \sigma_{b-1} \cdot \cdots \cdot \sigma_{a}$.
\end{definition}

To avoid confusion with exponents, we will always write exponents outside of parentheses. In this notation, then, a torus knot, $T(p, q)$, can be represented as the closure of a positive braid on $q$ strands, $\left(\Pi^{q-1}_1\right)^p$. Further, by introducing $n$ full twists into the first parallel $r$ strands at the bottom of this braid, a twisted torus knot, $K(p, q, r, n)$, can be represented as the closure of a braid on $q$ strands, $\left( \Pi^{q-1}_1 \right)^p \left( \Pi^{r-1}_1 \right)^{rn}$.

\subsection{Heegaard splittings}
A \emph{handlebody} is a 3-manifold with boundary that is homeomorphic to a boundary connected sum of solid tori. The \emph{genus} of a handlebody is the number of solid tori, or equivalently, the genus of the boundary surface of the handlebody. Given a 3-manifold, $M$, a \emph{Heegard splitting} is decomposition of $M$ along a closed, connected, orientable surface, called a \emph{Heegaard surface}, into two handlebodies. The genus of the Heegaard surface is called the \emph{genus} of the Heegaard splitting. Every closed, compact, connected, orientable 3-manifold has a Heegard splitting. The three-sphere has a unique Heegaard splitting of each genus (see \cite{WalHZ3S}).

A twisted torus knot, $K(p, q, r, n)$, has a natural position on the genus two Heegaard surface of $S^3$, as follows. The torus knot $T(p, q)$ sits on an unknotted torus, $T$. Let $D$ be a small disk in this torus intersecting $T(p, q)$ in $r$ adjacent sub-arcs. Let $T'$ be another unknotted torus, let $K'$ be a $r$ copies of $T(1, n)$ in $T'$, and let $D'$ be a small disk in $T'$ intersecting this link in $r$ adjacent sub-arcs. Then removing $D$ from $T$, $D'$ from $T'$, and identifying the resulting boundaries, keeping orientations consistent, results in $K(p, q, r, n)$ on the genus two Heegaard surface for $S^3$. See Fig. \ref{fig:ttkonF}.

\begin{figure}[h]
\includegraphics[scale=.55]{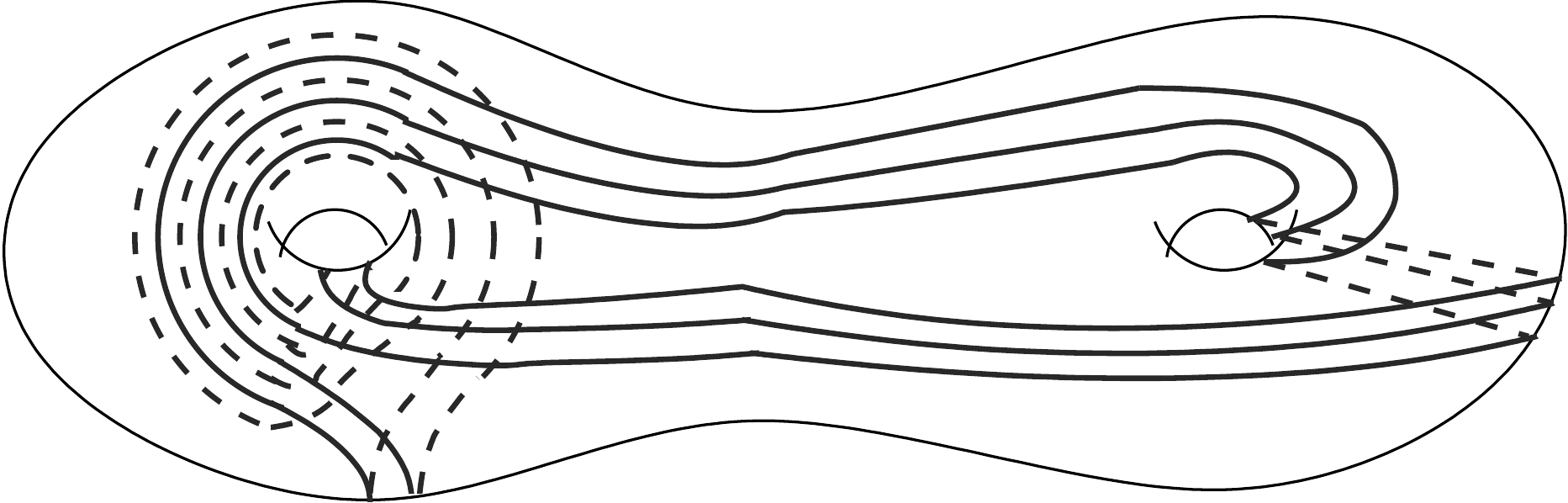}
\caption{$K(5, 2, 3, 1)$}
\label{fig:ttkonF}
\end{figure}

\section{Braids}
\label{section:Braids}
We will present two proofs of the main theorem. The approach is essentially the same in both, but our first proof will be algebraic, and our second proof will be purely geometric in nature. The two can be seen to complement one another. We will begin with several lemmas, as well as some discussion of how to interpret the statements of these lemmas in terms of braided strands. 

\begin{lemma}\label{lemma:Glesser}
If $a \leq b$, then $ \Prod{b}{a} \Prod{b+1}{a+1} \sigma_{a} ^{\pm 1} = \sigma_{b+1} ^{\pm 1} \Prod{b}{a} \Prod{b+1}{a+1}.$
\end{lemma}

\begin{proof}
If $a = b$, then the statement is equivalent to the braid relation. Henceforth, we will assume that $a < b$.
\begin{align*}
\Pi_a ^{b} \Pi_{a+1} ^{b+1} \sigma_a ^{\pm} = \Big(\sigma_b \sigma_{b-1} \sigma_{b-2} \cdots \sigma_{a+2} \sigma_{a+1} \sigma_a \Big) \Big(\sigma_{b+1} \sigma_{b} \cdots \sigma_{a+2} \sigma_{a+1}\Big) \sigma_a ^{\pm 1}.
\end{align*}
From $\Pi_{a+1}^{b+1}$, we can move $\sigma_{b+1}$ past all of the generators to the left of it except for $\sigma_{b}$. Therefore we have
\begin{align*}
\Big[ \sigma_b \Big( \sigma_{b+1} \Big) \sigma_{b-1} \sigma_{b-2} \cdots \sigma_{a+2} \sigma_{a+1} \sigma_a \Big] \Big[ \sigma_{b} \cdots \sigma_{a+2} \sigma_{a+1} \Big] \sigma_a ^{\pm 1}. 
\end{align*}

From $\Pi_{a+1}^{b+1}$, we can then move $\sigma_{b}$ past all of the generators to the left of it except for $\sigma_{b}$ and $\sigma_{b-1}$, to obtain
\begin{align*}
\Big[ \sigma_b \Big( \sigma_{b+1} \Big) \sigma_{b-1} \Big( \sigma_{b} \Big) \sigma_{b-2} \cdots \sigma_{a+2} \sigma_{a+1} \sigma_a \Big] \Big[ \sigma_{b-1} \cdots \sigma_{a+2} \sigma_{a+1} \Big] \sigma_a ^{\pm 1}. 
\end{align*}

Moving each of the remaining generators from $\Pi_{a+1}^{b+1}$ to the left as much as possible we get
\begin{align*}
\Big[ \sigma_b \Big( \sigma_{b+1} \Big) \sigma_{b-1} \Big( \sigma_{b} \Big) \sigma_{b-2} \Big( \sigma_{b-1} \Big) \sigma_{b-3} \cdots \sigma_{a+2} \Big( \sigma_{a+3} \Big) \sigma_{a+1} \Big( \sigma_{a+2} \Big) \sigma_{a} \Big( \sigma_{a+1} \Big) \sigma_a ^{\pm 1} \Big].  
\end{align*}

We can then apply the braid relation on $\sigma_a \sigma_{a+1} \sigma_a ^{\pm 1}$ to obtain
\begin{align*}
\Big[ \sigma_b \Big( \sigma_{b+1} \Big) \sigma_{b-1} \Big( \sigma_{b} \Big) \sigma_{b-2} \Big( \sigma_{b-1} \Big) \sigma_{b-3} \cdots \sigma_{a+2} \Big( \sigma_{a+3} \Big) \sigma_{a+1} \Big( \sigma_{a+2} \Big) \Big( \sigma_{a+1} ^{\pm 1} \sigma_{a} \sigma_{a+1} \Big) \Big] . 
\end{align*}

Applying the braid relation on $\sigma_{a+1} \sigma_{a+2} \sigma_{a+1} ^{\pm 1}$ we get
\begin{align*}
\Big[ \sigma_b \Big( \sigma_{b+1} \Big) \sigma_{b-1} \Big( \sigma_{b} \Big) \sigma_{b-2} \Big( \sigma_{b-1} \Big) \sigma_{b-3} \cdots \sigma_{a+2} \Big( \sigma_{a+3} \Big) \Big( \sigma_{a+2} ^{\pm 1} \Big( \sigma_{a+1} \Big) \sigma_{a+2} \Big) \sigma_{a} \sigma_{a+1} \Big] .
\end{align*}

We can continue to apply the braid relation $b-a-2$ more times to get
\begin{align*}
\Big[ \sigma_{b+1} ^{\pm 1} \Big( \sigma_{b} \Big) \sigma_{b+1} \Big( \sigma_{b-1} \Big) \sigma_{b} \Big( \sigma_{b-2} \Big) \sigma_{b-1} \cdots \sigma_{a+3} \Big( \sigma_{a+4} \Big) \sigma_{a+2} \Big( \sigma_{a+3} \Big)  \sigma_{a+1} \Big( \sigma_{a+2} \Big) \sigma_{a} \Big( \sigma_{a+1} \Big) \Big] .
\end{align*}

Moving $\sigma_{a+2}$ to the right past $\sigma_a$ we get
\begin{align*}
\Big[ \sigma_{b+1} ^{\pm 1} \Big( \sigma_{b} \Big) \sigma_{b+1} \Big( \sigma_{b-1} \Big) \sigma_{b} \Big( \sigma_{b-2} \Big) \sigma_{b-1} \cdots \sigma_{a+3} \Big( \sigma_{a+4} \Big) \sigma_{a+2} \Big( \sigma_{a+3} \Big)  \sigma_{a+1}  \sigma_{a} \Big( \sigma_{a+2} \Big) \Big( \sigma_{a+1} \Big) \Big] . 
\end{align*}

We can then move $\sigma_{a+3}$ to the right past $\sigma_{a+1}$ and $\sigma_a$ so that
\begin{align*}
\Big[ \sigma_{b+1} ^{\pm 1} \Big( \sigma_{b} \Big) \sigma_{b+1} \Big( \sigma_{b-1} \Big) \sigma_{b} \Big( \sigma_{b-2} \Big) \sigma_{b-1} \cdots \sigma_{a+3} \Big( \sigma_{a+4} \Big) \sigma_{a+2} \sigma_{a+1}  \sigma_{a} \Big( \sigma_{a+3} \Big) \Big( \sigma_{a+2} \Big) \Big( \sigma_{a+1} \Big) \Big] .
\end{align*}

Continuing in this way, we can move one of each of the generators $\sigma_{a+4}$ through $\sigma_{b+1}$ to the right to obtain

\belowdisplayskip = -12pt
\[\sigma_{b+1} ^{\pm 1} \Prod{b}{a}\Prod{b+1}{a+1}.\]
\end{proof}

We will find it useful in specific applications to have slightly more general versions of this statement, in the forms of the following Lemmas.

\begin{lemma}\label{sigma backward}
If $i \leq b-1$ then $\Prod{a}{1} \BigProd{a+1}{2}{a+b-1}{b} \sigma_i ^{\pm 1} = \sigma_{i+a} ^{\pm 1} \Prod{a}{1}\BigProd{a+1}{2}{a+b-1}{b}.$
\end{lemma}
\begin{proof}
Beginning from the expression on the left, move $\sigma_i$ to the left past as many generators as possible to get
$$\BigProd{a}{1}{a+b-1}{b} \sigma_i = \BigProd{a}{1}{a+i}{i+1} \sigma_i \BigProd{a+i+1}{i+2}{a+b-1}{b}.$$
By Lemma \ref{lemma:Glesser}, we can rewrite
$$\BigProd{a}{1}{a+i-2}{i-1} \Big( \Prod{a+i-1}{i} \Prod{a+i}{i+1} \sigma_i \Big) \BigProd{a+i+1}{i+2}{a+b-1}{b}$$ 
as
$$\BigProd{a}{1}{a+i-2}{i-1} \Big( \sigma_{i+a} \Prod{a+i-1}{i} \Prod{a+i}{i+1} \Big) \BigProd{a+i+1}{i+2}{a+b-1}{b}.$$
Since the index of $\sigma_{i+a}$ differs from all the other braid generators to its left by at least $2$, we can move $\sigma_{i+a}$ past all of the remaining generators to get
\belowdisplayskip = -12pt
\[\sigma_{i+a} \BigProd{a}{1}{a+b-1}{b}.\]

\end{proof}

In fact, we can interpret and understand the statement of Lemma \ref{sigma backward} as sliding a crossing (generator) along a group of strands that all pass over another group of strands. The result shifts the index of the generator. See Fig. \ref{figure:MovingGeneratorAlongStrands}.

\begin{figure}
\begin{center}
\begin{tikzpicture}[scale=0.75]
\braid[
 style all floors={fill=white},
 floor command={
 \fill (\floorsx,\floorsy) rectangle (\floorex, \floorey);
 \draw (\floorsx,\floorsy) -- (\floorex,\floorsy);
 },
 line width=2pt,
 number of strands=10
] (braid)at (2,0) s_2^{-1}-s_7^{-1}-s_9^{-1} s_1^{-1}-s_3^{-1}-s_6^{-1}-s_8^{-1} s_2^{-1}-s_4^{-1}-s_7^{-1}-s_9^{-1} s_3^{-1}-s_8^{-1} s_2^{-1};
\node[font=\Huge] at (6.5,-3) {\(=\)};
\draw[red,very thick] (3.5,-4.8) circle (5mm);
\draw[red,very thick] (10.5,-.7) circle (5mm);
\end{tikzpicture}
\end{center}
\caption{Moving $\sigma_i ^{\pm 1}$ over the first $r$ strands.}
\label{figure:MovingGeneratorAlongStrands}
\end{figure}
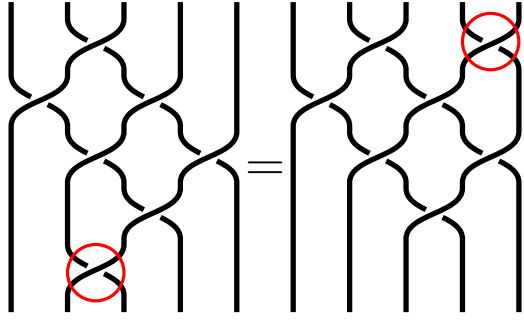

\begin{lemma}\label{lemma:ShiftBottomLeftToTopRight}
If $a \leq b$, then 
\begin{enumerate}
\item[(i)] \label{case:ShiftBottomLeftToTopRightShiftedIndices} \[ \Big(\Pi_a ^b \Pi_{a+1} ^{b+1} \cdots \Pi_{a+n-1} ^{b+n-1} \Pi_{a+n} ^{b+n} \Big) \Big(\Pi_{a} ^{a+n-1} \Big) = \Pi_{b+1} ^{b+n} \Big(\Pi_a ^b \Pi_{a+1} ^{b+1} \cdots \Pi_{a+n-2} ^{b+n-2} \Pi_{a+n-1} ^{b+n-1} \Pi_{a+n} ^{b+n} \Big), \] and 
\item[(ii)] \label{case:ShiftBottomLeftToTopRightSameIndices} \[ \Big(\Pi_a ^b \Pi_{a+1} ^{b+1} \cdots \Pi_{a+n-1} ^{b+n-1} \Pi_{a+n} ^{b+n} \Big) \Big(\Pi_{a} ^{a+n-1} \Big) = \Big(\Pi_{a} ^{b+n} \Big) \Big(\Pi_{a+1} ^{b+1} \Pi_{a+2} ^{b+2} \cdots \Pi_{a+n-1} ^{b+n-1} \Pi_{a+n} ^{b+n} \Big). \] 
\end{enumerate}
\end{lemma}
\begin{proof} 
Let $a \leq b$. We have
$$\Big(\Pi_a ^b \Pi_{a+1} ^{b+1} \cdots \Pi_{a+n-1} ^{b+n-1} \Pi_{a+n} ^{b+n} \Big) \Big(\Pi_{a} ^{a+n-1} \Big) = \Big(\Pi_a ^b \Pi_{a+1} ^{b+1} \cdots \Pi_{a+n-1} ^{b+n-1} \Pi_{a+n} ^{b+n} \Big) \Big(\sigma_{a+n-1}\Pi_{a} ^{a+n-2} \Big).$$
By Lemma \ref{lemma:Glesser}, we can rewrite the expression as,
$$\Big(\Pi_a ^b \Pi_{a+1} ^{b+1} \cdots \Pi_{a+n-2} ^{b+n-2}\sigma_{b+n} \Pi_{a+n-1} ^{b+n-1} \Pi_{a+n} ^{b+n} \Big) \Big(\Pi_{a} ^{a+n-2}\Big).$$
By the commutativity properties of the braid group, $\sigma_{b+n}$ can be moved past all of the generators to the left of it to get,
$$\sigma_{b+n} \Big(\Pi_a ^b \Pi_{a+1} ^{b+1} \cdots \Pi_{a+n-1} ^{b+n-1} \Pi_{a+n} ^{b+n} \Big) \Big(\Pi_{a} ^{a+n-2} \Big).$$
We can now rewrite $\Pi_a ^{a+n-2}$ as $\sigma_{a+n-2} \Pi_a ^{a+n-3}$. Therefore our expression becomes
$$\sigma_{b+n}\Big(\Pi_a ^b \Pi_{a+1} ^{b+1} \cdots \Pi_{a+n-1} ^{b+n-1} \Pi_{a+n} ^{b+n} \Big) \Big(\sigma_{a+n-2} \Pi_{a} ^{a+n-3} \Big).$$
We can then move $\sigma_{a+n-2}$ to get
$$\sigma_{b+n}\Big(\Pi_a ^b \Pi_{a+1} ^{b+1} \cdots \Pi_{a+n-2} ^{b+n-2} \Pi_{a+n-1} ^{b+n-1} \sigma_{a+n-2} \Pi_{a+n} ^{b+n} \Big) \Big(\Pi_{a} ^{a+n-3} \Big).$$
By Lemma \ref{lemma:Glesser}, we can rewrite the expression as,
$$\sigma_{b+n} \Big(\Pi_a ^b \Pi_{a+1} ^{b+1} \cdots \Pi_{a+n-3}^{b+n-3} \sigma_{b+n-1} \Pi_{a+n-2} ^{b+n-2} \Pi_{a+n-1} ^{b+n-1} \Pi_{a+n} ^{b+n} \Big) \Big(\Pi_{a} ^{a+n-3} \Big).$$
Moving $\sigma_{b+n-1}$ to the left we have,
$$\sigma_{b+n} \sigma_{b+n-1} \Big(\Pi_a ^b \Pi_{a+1} ^{b+1} \cdots \Pi_{a+n-2} ^{b+n-2} \Pi_{a+n-1} ^{b+n-1} \Pi_{a+n} ^{b+n} \Big) \Big(\Pi_{a} ^{a+n-3} \Big).$$
Continuing in this way $n-2$ times, our expression becomes
\[ \Pi_{b+1} ^{b+n} \Big(\Pi_a ^b \Pi_{a+1} ^{b+1} \cdots \Pi_{a+n-2} ^{b+n-2} \Pi_{a+n-1} ^{b+n-1} \Pi_{a+n} ^{b+n} \Big) ,\]
establishing the first equation, and combining the two left-most terms results in
\[  \Pi_{a} ^{b+n} \Big(\Pi_{a+1} ^{b+1} \cdots \Pi_{a+n-2} ^{b+n-2} \Pi_{a+n-1} ^{b+n-1} \Pi_{a+n} ^{b+n} \Big), \]
which establishes the second equation.
\end{proof}

Lemma \ref{lemma:ShiftBottomLeftToTopRight} can be considered a corollary of Lemma \ref{sigma backward}, as it is the result of shifting multiple crossings up a group of strands, though the algebraic presentation is peculiar. See Fig. \ref{figure:MovingProductAlongStrands}.

\begin{figure}[h]
\begin{center}
\begin{tikzpicture}[scale=0.75]
\braid[
 style all floors={fill=white},
 floor command={
 \fill (\floorsx,\floorsy) rectangle (\floorex, \floorey);
 \draw (\floorsx,\floorsy) -- (\floorex,\floorsy);
 },
 line width=2pt,
 number of strands=10
] (braid)at (2,0) s_2^{-1}- s_9^{-1} s_1^{-1}-s_3^{-1}-s_8^{-1} s_2^{-1}-s_4^{-1}-s_7^{-1} s_3^{-1}-s_6^{-1}-s_8^{-1} s_2^{-1}-s_7^{-1}-s_9^{-1} s_1^{-1}-s_8^{-1};
\node[font=\Huge] at (6.5,-3) {\(=\)};
\draw[red,very thick] (3,-5.3) circle (12mm);
\draw[red,very thick] (10,-1.3) circle (12mm);
\end{tikzpicture}
\end{center}
\caption{Moving $\Prod{a+n-1}{a}$ over the first $r$ strands.}
\label{figure:MovingProductAlongStrands}
\end{figure}
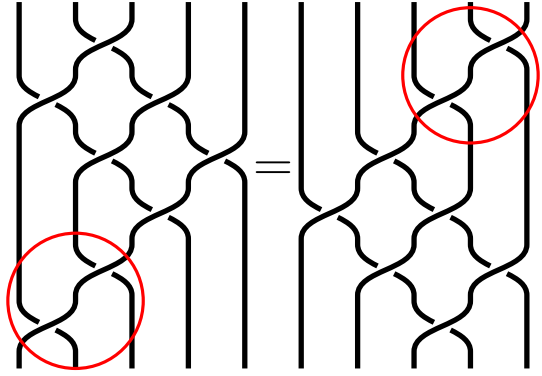

The following statement is a re-working of Lemma 3.4 from \cite{DolFPSTTK} by the second author, adapted to our different choices of convention and notation.
\begin{lemma}\label{lemma:Rework 3.4}
For $l < t \leq s$, $\Prod{s}{l} \sigma_t ^{\pm 1} = \sigma_{t-1} ^{\pm 1} \Prod{s}{l}$.
\end{lemma}

\begin{proof}
Rewriting the expression in terms of braid generators we have
\begin{align*}
\Prod{s}{l} \sigma_t ^{\pm 1} 
& = \Big( \sigma_{s} \sigma_{s-1} \cdots \sigma_{l+1} \sigma_{l} \Big) \sigma_{t} ^{\pm 1}.
\end{align*}

Moving $\sigma_{t} ^{\pm 1}$ past the first $t-l-1$ generators we get
$$ \sigma_{s} \sigma_{s-1} \cdots \Big( \sigma_{t+1} \sigma_{t} \sigma_{t-1} \sigma_{t} ^{\pm 1} \Big) \sigma_{t-2} \cdots \sigma_{l}.$$
By the braid relation $\sigma_{t} \sigma_{t-1} \sigma_{t} ^{\pm 1} = \sigma_{t-1} ^{\pm 1} \sigma_{t} \sigma_{t-1}.$ Therefore, 
$$\sigma_{s} \sigma_{s-1} \cdots \sigma_{t+1} \Big( \sigma_{t} \sigma_{t-1} \sigma_{t} ^{\pm 1} \Big) \sigma_{t-2} \cdots \sigma_{l} = \sigma_{s} \sigma_{s-1} \cdots \sigma_{t+1} \Big( \sigma_{t-1} ^{\pm 1} \sigma_{t} \sigma_{t-1} \Big) \sigma_{t-2} \cdots \sigma_{l}.$$
Moving $\sigma_{t-1} ^{\pm 1}$ past the remaining $s-t$ generators, we obtain  $\sigma_{t-1} ^{\pm 1} \Prod{s}{l}.$
\end{proof}

We state here an immediate corollary of Lemma \ref{lemma:Rework 3.4} that will be a convenient form for us.

\begin{lemma}\label{sigma forward}
If $b+1 \leq i \leq a+b-1$, then $\Prod{a}{1}\BigProd{a+1}{2}{a+b-1}{b} \sigma_i ^{\pm 1} = \sigma_{i-b} ^{\pm 1} \Prod{a}{1}\BigProd{a+1}{2}{a+b-1}{b}.$
\end{lemma}
\begin{proof}
Apply Lemma \ref{lemma:Rework 3.4} a total of $b$ times, noting that after each application, the new index of the generator will be between the top and bottom indices of the term immediately to its left.
\end{proof}

It is worth noting that Lemma \ref{sigma forward} is analogous to Lemma \ref{sigma backward}, only instead of sliding a crossing along a group of strands that run entirely \emph{above} another group of strands, it is sliding along a group running \emph{below} other strands.

The following Lemma is an interesting observation that a full twist can be factored into sub-twists, wherein the last $q-r$ strands wrap completely around the first $r$ strands, followed by a full twist within each of the groups of $q-r$ and $r$ strands, respectively. This is best appreciated by examining Fig. \ref{figure:FullTwistFactorization}.

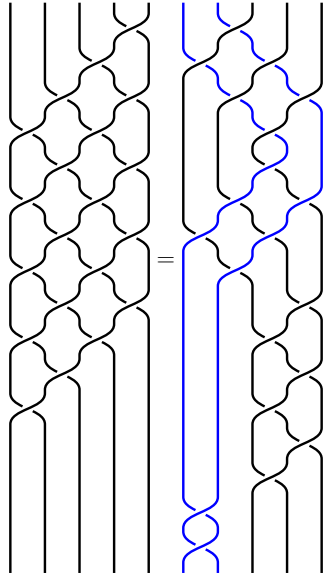
\begin{figure}[h]
\begin{center}
\begin{adjustbox}{max totalsize={3in}{3in},center}
\begin{tikzpicture}
\braid[
 style all floors={fill=white},
 floor command={
 \fill (\floorsx,\floorsy) rectangle (\floorex, \floorey);
 \draw (\floorsx,\floorsy) -- (\floorex,\floorsy);
 },
 line width=2pt,
 number of strands=10
 style strands={1}{},
 style strands={6}{blue},
 style strands={7}{blue},
] (braid) at (2,0) s_4^{-1}-s_7^{-1} 
s_3^{-1}-s_6^{-1}-s_8^{-1} 
s_2^{-1}-s_4^{-1}-s_7^{-1}-s_9^{-1} 
s_1^{-1}-s_3^{-1}-s_8^{-1} 
s_4^{-1}-s_2^{-1}-s_8^{-1}
s_3^{-1}-s_1^{-1}-s_7^{-1}-s_9^{-1} 
s_2^{-1}-s_6^{-1}-s_8^{-1}-s_4^{-1} 
s_1^{-1}-s_3^{-1}-s_7^{-1}
s_2^{-1}-s_4^{-1}-s_9^{-1} 
s_1^{-1}-s_3^{-1}-s_8^{-1} 
s_2^{-1}-s_9^{-1} 
s_1^{-1}-s_8^{-1} 
s_9^{-1} s_8^{-1}
s_6^{-1}
s_6^{-1};
\node[font=\Huge] at (6.5,-7.5) {\(=\)};

\end{tikzpicture}

\end{adjustbox}
\end{center}
\caption{Alternative factorization of one full twist.}
\label{figure:FullTwistFactorization}
\end{figure}

\begin{lemma}\label{full twist}
$\ProdPow{q-1}{1}{q} = \Big( \BigProd{r}{1}{q-1}{q-r} \Big) \Big( \BigProd{q-r}{1}{q-1}{r} \Big)  \ProdPow{q-1}{r+1}{q-r} \ProdPow{r-1}{1}{r}.$
\end{lemma}

\begin{proof}
We begin with the expression on the right,
$$\Big(\Pi_{1} ^r \Pi_2 ^{r+1} \cdots \Pi_{q-r} ^{q-1} \Big) \Big(\Pi_{1} ^{q-r} \Pi_{2} ^{q-r+1} \cdots \Pi_r ^{q-1} \Big) \Big(\Pi_{r+1} ^{q-1} \Big)^{q-r} \Big(\Pi_1 ^{r-1} \Big)^r.$$
We first explore how the terms in $(\Pi_{1} ^{q-r} \Pi_{2} ^{q-r+1} \cdots \Pi_r ^{q-1})$ and $\Pi_{r+1} ^{q-1}$ interact with each other. Rewriting $\Pi_{r+1} ^{q-1}$ as $\sigma_{q-1}\Pi_{r+1} ^{q-2}$ we have
$$\Big(\Pi_{1} ^{q-r} \Pi_{2} ^{q-r+1} \cdots \Pi_r ^{q-1} \Big) \Pi_{r+1} ^{q-1} = \Big(\Pi_{1} ^{q-r} \Pi_{2} ^{q-r+1} \cdots \Pi_r ^{q-1} \Big)\sigma_{q-1} \Pi_{r+1} ^{q-2}.$$
By Lemma \ref{lemma:Rework 3.4} we have,
$$\Big(\Pi_{1} ^{q-r} \Pi_{2} ^{q-r+1} \cdots \Pi_r ^{q-1} \Big)\sigma_{q-1} \Pi_{r+1} ^{q-2} = \Pi_{1} ^{q-r} \Pi_{2} ^{q-r+1} \cdots \Pi_{r-1} ^{q-2} \Big( \sigma_{q-2} \Pi_r ^{q-1} \Big) \Pi_{r+1} ^{q-2}.$$
Applying Lemma \ref{lemma:Rework 3.4} an additional $r-1$ times we get,
\begin{eqnarray*}
\Pi_{1} ^{q-r} \Pi_{2} ^{q-r+1} \cdots \Pi_{r-1} ^{q-2} \Big( \sigma_{q-2} \Pi_r ^{q-1} \Big) \Pi_{r+1} ^{q-2} & = & \Pi_{1} ^{q-r} \Pi_{2} ^{q-r+1} \cdots \Pi_{r-2} ^{q-3} \Big( \sigma_{q-3} \Pi_{r-1} ^{q-2} \Big)  \Pi_r ^{q-1}\Pi_{r+1} ^{q-2} \\
& = & \Pi_{1} ^{q-r} \Pi_{2} ^{q-r+1} \cdots  \Pi_{r-3} ^{q-4} \Big( \sigma_{q-4} \Pi_{r-2} ^{q-3} \Big) \Pi_{r-1} ^{q-2}  \Pi_r ^{q-1}\Pi_{r+1} ^{q-2} \\
& \vdots & \\
& = & \sigma_{q-r-1} \Pi_{1} ^{q-r} \Pi_{2} ^{q-r+1} \cdots  \Pi_{r-3} ^{q-4} \Pi_{r-2} ^{q-3} \Pi_{r-1} ^{q-2}  \Pi_r ^{q-1}\Pi_{r+1} ^{q-2}.
\end{eqnarray*} 

We can now rewrite $\Prod{q-2}{r+1}$ as $\sigma_{q-2} \Prod{q-3}{r+1}$ to get
$$\sigma_{q-r-1} \Pi_{1} ^{q-r} \Pi_{2} ^{q-r+1} \cdots  \Pi_r ^{q-1}\Pi_{r+1} ^{q-2} = \Big( \sigma_{q-r-1} \Pi_{1} ^{q-r} \Pi_{2} ^{q-r+1} \cdots  \Pi_{r-1} ^{q-2}  \Pi_r ^{q-1} \Big) \sigma_{q-2} \Pi_{r+1} ^{q-3} . $$

As before, $r$ applications of Lemma \ref{lemma:Rework 3.4} gives
\begin{eqnarray*}
\Big( \sigma_{q-r-1} \Pi_{1} ^{q-r} \Pi_{2} ^{q-r+1} \cdots    \Pi_r ^{q-1} \Big) \sigma_{q-2} \Pi_{r+1} ^{q-3} & = & \sigma_{q-r-1} \Pi_{1} ^{q-r}  \cdots  \Pi_{r-1} ^{q-2}  \Big( \sigma_{q-3} \Pi_r ^{q-1} \Big) \Pi_{r+1} ^{q-3} \\
& = &  \sigma_{q-r-1} \Pi_{1} ^{q-r} \cdots  \Pi_{r-2} ^{q-3} \Big( \sigma_{q-3} \Pi_{r-1} ^{q-2} \Big) \Pi_r ^{q-1} \Pi_{r+1} ^{q-3}\\
& \vdots & \\
& = & \sigma_{q-r-1} \sigma_{q-r-2} \Pi_{1} ^{q-r} \cdots  \Pi_{r-1} ^{q-2}  \Pi_r ^{q-1}\Pi_{r+1} ^{q-3}.
\end{eqnarray*}

Each of the remaining generators of $\Pi_{r+1} ^{q-3}$ can be moved over in precisely this way, so we will obtain 
$\Big( \sigma_{q-r-1} \sigma_{q-r-2} \cdots \sigma_{2} \sigma_{1} \Big) \Pi_{1} ^{q-r} \Pi_{2} ^{q-r+1} \cdots \Pi_{r-1} ^{q-2}  \Pi_r ^{q-1}$,
which can then be simplified to $\Prod{q-r-1}{1}\BigProd{q-r}{1}{q-1}{q-r}.$
We can now use this entire relation on the $q-r$ copies of the expression $\Pi_{r+1} ^{q-1}$ to get
$$\Pi_{1} ^{q-r} \Pi_{2} ^{q-r+1} \cdots \Pi_{r-1} ^{q-2}  \Pi_r ^{q-1} \ProdPow{q-1}{r+1}{q-r} = \ProdPow{q-r-1}{1}{q-r} \Pi_{1} ^{q-r} \Pi_{2} ^{q-r+1} \cdots \Pi_{r-1} ^{q-2}  \Pi_r ^{q-1}.$$
Therefore, we can express 
our original braid word as
$$\Big(\Pi_{1} ^r \Pi_2 ^{r+1} \cdots \Pi_{q-r} ^{q-1} \Big) \Big(\Pi_{1} ^{q-r-1} \Big)^{q-r} \Big(\Pi_{1} ^{q-r} \Pi_{2} ^{q-r+1} \cdots \Pi_r ^{q-1} \Big) \Big(\Pi_1 ^{r-1} \Big)^r.$$

We next look at the interaction of the terms in $(\Pi_{1} ^r \Pi_2 ^{r+1} \cdots \Pi_{q-r} ^{q-1})$ and $\Big( \Pi_{1} ^{q-r-1} \Big)^{q-r}$. Let's first peel off one copy of $\Big( \Pi_{1} ^{q-r-1} \Big)$, so
\[
\Big(\Pi_{1} ^r \Pi_2 ^{r+1} \cdots \Pi_{q-r} ^{q-1} \Big) \Big( \Pi_{1} ^{q-r-1} \Big)^{q-r}  =  \Big(\Pi_{1} ^r \Pi_2 ^{r+1} \cdots \Pi_{q-r} ^{q-1} \Big)\Big( \Pi_{1} ^{q-r-1} \Big) \Big( \Pi_{1} ^{q-r-1} \Big)^{q-r-1}.
\]

We can now apply Lemma \ref{lemma:ShiftBottomLeftToTopRight}, (\ref{case:ShiftBottomLeftToTopRightShiftedIndices}) to get,
\[\Big(\Pi_{1} ^r \Pi_2 ^{r+1} \cdots \Pi_{q-r} ^{q-1} \Big)\Big( \Pi_{1} ^{q-r-1} \Big) \Big( \Pi_{1} ^{q-r-1} \Big)^{q-r-1} = \Pi_{1} ^{q-1} \Big( \Pi_2 ^{r+1} \cdots \Pi_{q-r} ^{q-1} \Big) \Big( \Pi_{1} ^{q-r-1} \Big)^{q-r-1}.\]
In fact, taking $a = 2$, $b = r + 1$, and $n = q-r-2$ 
in Lemma \ref{lemma:ShiftBottomLeftToTopRight}, (\ref{case:ShiftBottomLeftToTopRightShiftedIndices}), we now get
\begin{align*}
\Pi_{1} ^{q-1} \Big( \Pi_2 ^{r+1} \cdots \Pi_{q-r} ^{q-1} \Big) \Big( \Pi_{1} ^{q-r-1} \Big)^{q-r-1} & = & \Pi_{1} ^{q-1} \Big( \Pi_2 ^{r+1} \cdots \Pi_{q-r} ^{q-1} \Big) \Big( \Pi_{1} ^{q-r-1} \Big) \Big( \Pi_{1} ^{q-r-1} \Big)^{q-r-2} \\
& = & \Big(\Pi_{1} ^{q-1} \Big) \Big( \Pi_2 ^{r+1} \cdots \Pi_{q-r} ^{q-1} \Big) \Big(\Prod{q-r-1}{2} \sigma_{1} \Big) \Big(\Pi_{1} ^{q-r-1} \Big)^{q-r-2} \\
& = & \Big(\Pi_{1} ^{q-1} \Pi_2 ^{q-1} \Big) \Big( \Prod{r+2}{3} \cdots \Pi_{q-r} ^{q-1} \Big) \Big(\sigma_{1} \Big) \Big(\Pi_{1} ^{q-r-1} \Big)^{q-r-2} \\
& = & \Big(\Pi_{1} ^{q-1} \Pi_2 ^{q-1} \Big) \Big( \Prod{r+2}{3} \cdots \Pi_{q-r} ^{q-1} \Big) \Big(\sigma_{1} \Big) \Big(\Pi_{1} ^{q-r-1} \Big)^{q-r-2}.
\end{align*}

We can then use commutative properties of the braid group to move $\sigma_{1}$ to the left to get
$$\Big(\Pi_{1} ^{q-1} \Pi_2 ^{q-1} \sigma_{1} \Big) \Big( \Prod{r+2}{3} \cdots \Pi_{q-r} ^{q-1} \Big) \Big(\Pi_{1} ^{q-r-1} \Big)^{q-r-2} = \Big(\Pi_{1} ^{q-1} \Big)^{2} \Big( \Pi_{3} ^{r+2} \cdots \Pi_{q-r} ^{q-1} \Big) \Big(\Pi_{1} ^{q-r-1} \Big)^{q-r-2}.$$
Continuing in this way $(q-r-2)$ more times, we can express the word on the right as
$(\Pi_{1} ^{q-1})^{q-r}.$
By similar reasoning, we can express $(\Pi_{1} ^{q-r} \Pi_{2} ^{q-r+1} \cdots \Pi_r ^{q-1})(\Pi_1 ^{r-1})^r$ as 
$(\Pi_{1} ^{q-1})^r.$
Therefore,
\belowdisplayskip = -12pt
\begin{align*}
\Big(\Pi_{1} ^r \Pi_2 ^{r+1} \cdots \Pi_{q-r} ^{q-1} \Big) \Big(\Pi_{1} ^{q-r} \Pi_{2} ^{q-r+1} \cdots \Pi_r ^{q-1} \Big) \Big(\Pi_{q-r} ^{q-1} \Big)^{q-r} \Big(\Pi_1 ^{r-1} \Big)^r &= \Big(\Pi_{1} ^{q-1} \Big)^{q-r} \Big(\Pi_{1} ^{q-1} \Big)^{r}\\
&= \Big(\Pi_1 ^{q-1} \Big)^q. 
\end{align*}
\end{proof}

The heart of our proof and the ability to introduce \emph{additional} cases of positive twisted torus knots beyond those discovered in \cite{DolFPSTTK} is the following Lemma. While \cite{DolFPSTTK} uses positive full twists in the torus knot to cancel negative twists among a subset of strands, if there are not enough of these positive full twists, then not all of the negative crossings may be eliminated. However, if there are enough \emph{partial} twists in the positive direction remaining from the torus knot, then some cancellations may still be possible. See Fig. \ref{figure:CancelingNegativeRTwists}.

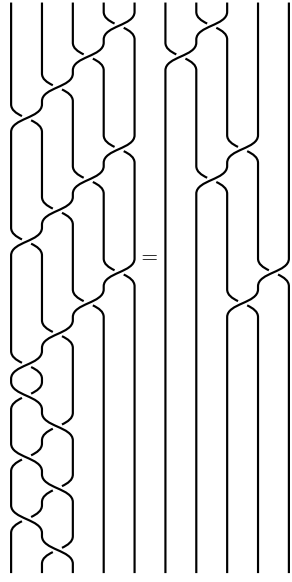
\begin{figure}[h]
\begin{center}
\begin{adjustbox}{max totalsize={3in}{3in},center}
\begin{tikzpicture}
\braid[
 style all floors={fill=white},
 floor command={
 \fill (\floorsx,\floorsy) rectangle (\floorex, \floorey);
 \draw (\floorsx,\floorsy) -- (\floorex,\floorsy);
 },
 line width=2pt,
 number of strands=10
] (braid)at (2,0) s_4^{-1}-s_7^{-1} s_3^{-1}-s_6^{-1} s_2^{-1} s_1^{-1} s_4^{-1}-s_8^{-1} s_3^{-1}-s_7^{-1} s_2^{-1} s_1^{-1} s_4^{-1}-s_9^{-1} s_3^{-1}-s_8^{-1} s_2^{-1} s_1^{-1} s_1 s_2 s_1 s_2 s_1 s_2;
\node[font=\Huge] at (6.5,-8.3) {\(=\)};
\end{tikzpicture}
\end{adjustbox}
\end{center}
\caption{The negative full twists on $r$ strands at the bottom can be systematically undone using the positive crossings above, even though they do not constitute even one full twist of the $q$ strands.}
\label{figure:CancelingNegativeRTwists}
\end{figure}

\begin{lemma}\label{pi relation}
If $r < q$, then $\ProdPow{q-1}{1}{r} \ProdPow{r-1}{1}{-r} = \Prod{q-r}{1} \Prod{q-r+1}{2} \cdots \Prod{q-2}{r-1} \Prod{q-1}{r}$.

\end{lemma}

\begin{proof}
By peeling off one term from each of the exponentiated terms, and cancelling some of the positive and negative generators we find that
\begin{align*}
\ProdPow{q-1}{1}{r} \ProdPow{r-1}{1}{-r} & = & \ProdPow{q-1}{1}{r-1} \Prod{q-1}{1} \ProdPow{r-1}{1}{-1}  \ProdPow{r-1}{1}{-r+1} \\
& = &  \ProdPow{q-1}{1}{r-1} \Prod{q-1}{r}   \ProdPow{r-1}{1}{-r+1}. 
\end{align*}
Once again, peeling one term from the left and right terms, we get
\begin{align*}
 \ProdPow{q-1}{1}{r-2} \ProdPow{q-1}{1}{}   \Prod{q-1}{r}  \ProdPow{r-1}{1}{-1}   \ProdPow{r-1}{1}{-r+2}.
\end{align*}
Express $\ProdPow{r-1}{1}{-1}$ as $\ProdPow{r-2}{1}{-1} \sigma_{r-1} ^{-1}$, and note that the indices of $\ProdPow{r-2}{1}{-1}$ are at least 2 lower than all of the generators of $\Prod{q-1}{r}$. Then we can move them all to the left and cancel them with the generators in $\Prod{q-1}{1}$ to get
\begin{align*}
\ProdPow{q-1}{1}{r-2} \ProdPow{q-1}{1}{}   \Prod{q-1}{r}  \ProdPow{r-1}{1}{-1}   \ProdPow{r-1}{1}{-r+2} & = \\ 
\ProdPow{q-1}{1}{r-2} \ProdPow{q-1}{1}{}   \Prod{q-1}{r}  \ProdPow{r-2}{1}{-1} \sigma_{r-1} ^{-1}   \ProdPow{r-1}{1}{-r+2} & = \\
\ProdPow{q-1}{1}{r-2} \ProdPow{q-1}{r-1}{}   \Prod{q-1}{r}  \sigma_{r-1} ^{-1}   \ProdPow{r-1}{1}{-r+2}.
\end{align*}
If we express $\Prod{q-1}{r-1}$ as $\sigma_{q-1}\Prod{q-2}{r-1}$ we can then apply Lemma \ref{lemma:Glesser} to get
\begin{align*}
\ProdPow{q-1}{1}{r-2} \ProdPow{q-1}{r-1}{}   \Prod{q-1}{r}  \sigma_{r-1} ^{-1}   \ProdPow{r-1}{1}{-r+2} & = \\
\ProdPow{q-1}{1}{r-2} \sigma_{q-1}\ProdPow{q-2}{r-1}{}   \Prod{q-1}{r}  \sigma_{r-1} ^{-1}   \ProdPow{r-1}{1}{-r+2} & = \\
\ProdPow{q-1}{1}{r-2}  \sigma_{q-1} \sigma_{q-1}^{-1} \ProdPow{q-2}{r-1}{}   \Prod{q-1}{r}  \ProdPow{r-1}{1}{-r+2} & = \\ 
\ProdPow{q-1}{1}{r-2} \ProdPow{q-2}{r-1}{}   \Prod{q-1}{r}  \ProdPow{r-1}{1}{-r+2}.
\end{align*}
Once again, peeling off a term from the left and right, we obtain
\begin{align*}
\ProdPow{q-1}{1}{r-3} \ProdPow{q-1}{1}{} \Prod{q-2}{r-1}   \Prod{q-1}{r} \ProdPow{r-1}{1}{-1} \ProdPow{r-1}{1}{-r+3}.
\end{align*}
Express $\ProdPow{r-1}{1}{-1}$ as $\ProdPow{r-3}{1}{-1} \sigma_{r-2} ^{-1} \sigma_{r-1} ^{-1}$, and note that the indices of $\ProdPow{r-3}{1}{-1}$ are at least 2 lower than all of the generators of $\Prod{q-2}{r-1} \Prod{q-1}{r}$. We can move them all to the left and cancel them with the generators in $\Prod{q-1}{1}$ to get
\begin{align*}
\ProdPow{q-1}{1}{r-3} \ProdPow{q-1}{1}{} \Prod{q-2}{r-1}   \Prod{q-1}{r} \ProdPow{r-1}{1}{-1} \ProdPow{r-1}{1}{-r+3} & = \\
\ProdPow{q-1}{1}{r-3} \ProdPow{q-1}{1}{} \Prod{q-2}{r-1}   \Prod{q-1}{r} \ProdPow{r-3}{1}{-1} \sigma_{r-2} ^{-1} \sigma_{r-1} ^{-1} \ProdPow{r-1}{1}{-r+3} & = \\
\ProdPow{q-1}{1}{r-3} \ProdPow{q-1}{r-2}{} \Prod{q-2}{r-1}   \Prod{q-1}{r} \sigma_{r-2} ^{-1} \sigma_{r-1} ^{-1} \ProdPow{r-1}{1}{-r+3}.
\end{align*}
Since the index of $\sigma_{r-2} ^{-1}$ is 2 lower than all of the generators of $\Prod{q-1}{r}$, we can move $\sigma_{r-2}$ to the left to get
\begin{align*}
\ProdPow{q-1}{1}{r-3} \ProdPow{q-1}{r-2}{} \Prod{q-2}{r-1} \Big( \sigma_{r-2} ^{-1} \Big) \Prod{q-1}{r}  \sigma_{r-1} ^{-1} \ProdPow{r-1}{1}{-r+3}.
\end{align*}
Rewriting $\Prod{q-1}{r-2}$ as $\sigma_{q-1} \sigma_{q-2} \Prod{q-3}{r-2}$ we can apply Lemma \ref{lemma:Glesser} to $\ProdPow{q-3}{r-2}{} \Prod{q-2}{r-1} \sigma_{r-2} ^{-1}$ to get
\begin{align*}
\ProdPow{q-1}{1}{r-3} \Prod{q-1}{r-2} \Prod{q-2}{r-1} \Big( \sigma_{r-2} ^{-1} \Big) \Prod{q-1}{r}  \sigma_{r-1} ^{-1} \ProdPow{r-1}{1}{-r+3} & = \\
\ProdPow{q-1}{1}{r-3} \sigma_{q-1} \sigma_{q-2} \Big( \Prod{q-3}{r-2} \Prod{q-2}{r-1} \sigma_{r-2} ^{-1} \Big) \Prod{q-1}{r}  \sigma_{r-1} ^{-1} \ProdPow{r-1}{1}{-r+3} & = \\
\ProdPow{q-1}{1}{r-3} \sigma_{q-1} \sigma_{q-2} \Big( \sigma_{q-2} ^{-1} \Big) \Prod{q-3}{r-2} \Prod{q-2}{r-1}  \Prod{q-1}{r}  \sigma_{r-1} ^{-1} \ProdPow{r-1}{1}{-r+3} & =\\
\ProdPow{q-1}{1}{r-3} \sigma_{q-1} \Prod{q-3}{r-2} \Prod{q-2}{r-1}  \Prod{q-1}{r}  \sigma_{r-1} ^{-1} \ProdPow{r-1}{1}{-r+3}.
\end{align*}
Now we can apply Lemma \ref{lemma:Glesser} to $\Prod{q-2}{r-1} \Prod{q-1}{r} \sigma_{r-1} ^{-1}$ to obtain
\begin{align*}
\ProdPow{q-1}{1}{r-3} \sigma_{q-1} \Prod{q-3}{r-2} \Big( \sigma_{q-1} ^{-1} \Big) \Prod{q-2}{r-1}  \Prod{q-1}{r} \ProdPow{r-1}{1}{-r+3}.
\end{align*}
Moving $\sigma_{q-1} ^{-1}$ past $\Prod{q-3}{r-2}$, our expression becomes
\begin{align*}
\ProdPow{q-1}{1}{r-3} \sigma_{q-1} \Big( \sigma_{q-1} ^{-1} \Big) \Prod{q-3}{r-2} \Prod{q-2}{r-1}  \Prod{q-1}{r} \ProdPow{r-1}{1}{-r+3} & = \\
\ProdPow{q-1}{1}{r-3} \Prod{q-3}{r-2} \Prod{q-2}{r-1}  \Prod{q-1}{r} \ProdPow{r-1}{1}{-r+3}
\end{align*}
Continuing in this way $r-3$ more times we have, $\Prod{q-r}{1} \Prod{q-r+1}{2} \cdots \Prod{q-2}{r-1} \Prod{q-1}{r}$.

\end{proof}

We are now is a position to prove our Main Theorem.

\MainTheorem*

One particularly nice property of knots formed as the closure of a positive braid is that such knots are known to be \emph{fibered}, by work of Stallings \cite{StaCFKL}. A knot $K$ in $S^{3}$ is fibered if $S^{3} - K$ is homeomorphic to $(F \times I)/f$, where $F$ is the interior of a Seifert surface for $K$ and the map $f:F \times \left\{ 0 \right\} \rightarrow F \times \left\{ 1 \right\}$ is a homeomorphism.

\begin{corollary}
\label{corollary:Fibered}
The knots in Theorem \ref{theorem:MainTheorem} are fibered.
\end{corollary}

\begin{proof}[First Proof of Theorem \ref{theorem:MainTheorem}] If $n \leq k$, then $nq \leq kq < p$, and the result was proven in \cite{DolFPSTTK}. We limit our proof, then, to the case that $n = k+1$.

The twisted torus knot $K(p,q,r,-(k+1))$ can be written as the closure of the braid word
\begin{align*}
\ProdPow{q-1}{1}{p} \ProdPow{r-1}{1}{-(k+1)r} =  \ProdPow{q-1}{1}{kq+e} \ProdPow{r-1}{1}{-(k+1)r} =\ProdPow{q-1}{1}{e} \left [ \ProdPow{q-1}{1}{q} \right ]^{k}  \ProdPow{r-1}{1}{-(k+1)r}.
\end{align*}
Peeling one term off of $\Big[ \Big( \Prod{q-1}{1} \Big)^{q} \Big]^{k}$ we can then apply Lemma \ref{full twist} to write the braid word as,
$$\ProdPow{q-1}{1}{e} \left [ \ProdPow{q-1}{1}{q} \right ]^{k-1} \Big( \Big( \BigProd{r}{1}{q-1}{q-r} \Big)\Big( \BigProd{q-r}{1}{q-1}{r} \Big)\Big( \Prod{q-1}{r+1}\Big)^{q-r} \ProdPow{r-1}{1}{r}\Big) \ProdPow{r-1}{1}{-(k+1)r}$$
which simplifies to 
$$\ProdPow{q-1}{1}{e} \left [ \ProdPow{q-1}{1}{q} \right ]^{k-1} \Big( \BigProd{r}{1}{q-1}{q-r} \Big)\Big( \BigProd{q-r}{1}{q-1}{r} \Big)\Big( \Prod{q-1}{r+1}\Big)^{q-r} \ProdPow{r-1}{1}{-(k+1)r+r}.$$
We will move the term $\ProdPow{r-1}{1}{-(k+1)r+r}$ from the right, through the next three terms on the left, tracking the effect on the indices of the generators in the product as we do so.

First, the indices of the term on the far right are at least two lower than any of the indices in the term to its immediate left, so these two terms commute. We can then use Lemma \ref{lemma:ShiftBottomLeftToTopRight} (\ref{case:ShiftBottomLeftToTopRightSameIndices}), to move this term through the next term to the left, shifting its indices up by $(q-r)$. Our expression becomes
$$\ProdPow{q-1}{1}{e} \left [ \ProdPow{q-1}{1}{q} \right ]^{k-1} \Big( \BigProd{r}{1}{q-1}{q-r} \Big) \Big( \Prod{q-1}{q-r+1} \Big)^{-(k+1)r+r} \Big( \BigProd{q-r}{1}{q-1}{r} \Big)\Big( \Prod{q-1}{r+1}\Big)^{q-r}.$$
Now, the indices are such that we can apply Lemma \ref{sigma forward} to move this term through the next term to the left, shifting its indices down by $(q-r)$. The net effect of all of these moves creates the braid word
\[\ProdPow{q-1}{1}{e} \left [ \ProdPow{q-1}{1}{q} \right ]^{k-1} \Big( \Prod{r-1}{1} \Big)^{-(k+1)r+r} \Big( \BigProd{r}{1}{q-1}{q-r} \Big) \Big( \BigProd{q-r}{1}{q-1}{r} \Big)\Big( \Prod{q-1}{r+1}\Big)^{q-r}.\]

Now, let $\alpha = \Big( \BigProd{r}{1}{q-1}{q-r} \Big) \Big( \BigProd{q-r}{1}{q-1}{r} \Big)\Big( \Prod{q-1}{r+1}\Big)^{q-r} $. By Lemma \ref{full twist}, we can factor another single full twist from $\Big[\ProdPow{q-1}{1}{q} \Big]^{k-1}$ to write the braid word as

\[ \ProdPow{q-1}{1}{e} \left [ \ProdPow{q-1}{1}{q} \right ]^{k-1} \ProdPow{r-1}{1}{-(k+1)r+r} \alpha \]
\[ = \ProdPow{q-1}{1}{e} \left [ \ProdPow{q-1}{1}{q} \right ]^{k-2} \Big( \Big( \BigProd{r}{1}{q-1}{q-r} \Big)\Big( \BigProd{q-r}{1}{q-1}{r} \Big)\Big( \Prod{q-1}{r+1}\Big)^{q-r} \ProdPow{r-1}{1}{r}\Big) \ProdPow{r-1}{1}{-(k+1)r+r} \alpha \]
\[ = \ProdPow{q-1}{1}{e} \left [ \ProdPow{q-1}{1}{q} \right ]^{k-2} \Big( \Big( \BigProd{r}{1}{q-1}{q-r} \Big)\Big( \BigProd{q-r}{1}{q-1}{r} \Big)\Big( \Prod{q-1}{r+1}\Big)^{q-r} \Big) \ProdPow{r-1}{1}{-(k+1)r+2r} \alpha. \]

By Lemmas \ref{lemma:ShiftBottomLeftToTopRight} (\ref{case:ShiftBottomLeftToTopRightSameIndices}) and then \ref{sigma forward}, as above, our expression becomes

\[\ProdPow{q-1}{1}{e} \left [ \ProdPow{q-1}{1}{q} \right ]^{k-2} \Big( \Big( \BigProd{r}{1}{q-1}{q-r} \Big)\Big( \BigProd{q-r}{1}{q-1}{r} \Big)\Big( \Prod{q-1}{r+1}\Big)^{q-r} \Big) \ProdPow{r-1}{1}{-(k+1)r+2r} \alpha \]
\[ = \ProdPow{q-1}{1}{e} \left [ \ProdPow{q-1}{1}{q} \right ]^{k-2}  \ProdPow{r-1}{1}{-(k+1)r+2r} \Big( \Big( \BigProd{r}{1}{q-1}{q-r} \Big)\Big( \BigProd{q-r}{1}{q-1}{r} \Big)\Big( \Prod{q-1}{r+1}\Big)^{q-r} \Big) \alpha \]
\[ = \ProdPow{q-1}{1}{e} \left [ \ProdPow{q-1}{1}{q} \right ]^{k-2} \ProdPow{r-1}{1}{-(k+1)r+2r} \alpha ^{2}.\]

After repeating this process $k-2$ more times we end up with
$$\ProdPow{q-1}{1}{e} \ProdPow{r-1}{1}{-(k+1)r+kr} \alpha^{k},$$
which simplifies to
$$\ProdPow{q-1}{1}{e} \ProdPow{r-1}{1}{-r} \alpha^{k}.$$
Since $e \geq r$, by Lemma \ref{pi relation} we finally arrive at
$$\ProdPow{q-1}{1}{e-r} \Big( \Prod{q-r}{1} \BigProd{q-r+1}{2}{q-2}{r-1} \Prod{q-1}{r} \Big) \Big[ \Big( \BigProd{r}{1}{q-1}{r+1} \Big) \Big( \BigProd{q-r}{1}{q-1}{r} \Big)\Big( \Prod{q-1}{r+1}\Big)^{q-r} \Big]^{k}$$
which is evidently positive.
\end{proof}

For a more directly geometric interpretation, we also offer the following proof.
\begin{proof}[Second Proof of Theorem \ref{theorem:MainTheorem}]
We will refer to a \emph{partial twist} (on $q$ strands) in a braid as a portion of the braid in which one strand crosses once over every other strand, corresponding to the word $\Pi^{q-1}_{1}$. The torus knot $T(p,q)$, then, can be considered as the closure of a braid in $B_q$ with $p$ partial twists. We consider the twisted torus knot  $K(p,q,r,n)$ as the closure of a braid in $B_q$ by adding $n$ full twists to $r$ adjacent strands in the braid representing $T(p,q)$. The choice of $r$ adjacent strands does not change the knot, so we use the first $r$ strands for ease of computation. For simplicity, the braid is drawn as a link surgery diagram, as shown in Fig. \ref{figure: og braid}. Here, a box with integer $m$ represents the braid word $\left( \Pi^{q-1}_1 \right)^m$. Since $p = kq + e$, this braid will have $e$ partial twists and $k$ full twists on all $q$ strands followed by $-n$ full twists on the first $r$ strands. Thus, we have Fig. \ref{figure: all link surgery}.

For any braid in $B_q$ and $r<q$, a full twist on all $q$ strands can be represented instead by a braid where the first $r$ strands pass over and then under the remaining $q-r$ strands and the groups of the first $r$ strands and the remaining $q-r$ strands each have a full twist on that group of strands. The braid theoretic statement of this is given in Lemma \ref{full twist}.

In $K(p,q,r,-n)$, since there are $k$ full twists on all $q$ strands, we can transform each full twist to obtain Fig. \ref{figure: all twists}. Since the first $r$ strands and the remaining $q-r$ strands stay in the same positions through a full twist, the link components indicating full twists can be pushed to the top of the chain of twists between the group of $r$ strands and the group of $q-r$ strands. The link surgeries on the first $r$ strands are combined into a surgery on a single component to obtain Fig. \ref{figure: surgery pushed up}.

In Fig. \ref{figure: surgery pushed up}, all crossings are positive except, possibly, in the region where $1/(n-k)$ surgery on a link component is indicated, so we may consider only the parts of the diagram above the dashed line. If $n \leq k$, then there are already no negative twists in this diagram, and we have recovered a proof of Theorem \ref{theorem:OriginalDoleshal}. If $n = k + 1$, then since $r \leq e$, we may use Lemma \ref{pi relation} to redraw the portion above the dashed line as Fig. \ref{figure: top of braid}. Now we have a braid with all positive crossings. 
\end{proof}

\begin{figure}
\begin{center}
\begin{tikzpicture}
\draw (0,0) -- (0,1);
\draw (.25,0) -- (.25,1);
\draw (.5,0) -- (.5,1);
\node[draw=none] at (.9,.5) {$\cdots$};
\draw (1.25, 0) -- (1.25,1);
\draw (1.5, 0) -- (1.5,1);
\draw (1.75, 0) -- (1.75,1);
\node[draw=none] at (.9,.5) {$\cdots$};
\node[draw=none] at (2.15,.5) {$\cdots$};
\draw (2.5,0) -- (2.5,1);
\draw (2.75,0) --(2.75,1);
\node[draw=none] at (1.4,-.5) {$kq + e$};
\draw (-.1, -1) rectangle (2.85, 0); 
\draw (0, -1) -- (0,-1.5);
\draw (0.25, -1) -- (0.25,-1.5);
\draw (.5,-1) -- (.5,-1.5);
\draw (1.75,-1) --(1.75,-2.5);
\node[draw=none] at (.9,-1.25) {$\cdots$};
\draw (1.25, -1) -- (1.25,-1.5);
\draw (1.5,-1.65) -- (1.5,-2.5);
\draw (2.5,-1) -- (2.5,-2.5);
\draw (2.75,-1) --(2.75,-2.5);
\draw (1.5,-1) -- (1.5,-1.5);
\draw (0,-1.65) -- (0,-2.5);
\draw (.25,-1.7) -- (.25,-2.5);
\draw (.5,-1.7) -- (.5,-2.5);
\node[draw=none] at (.9,-2.1) {$\cdots$};
\draw (1.25,-1.7) -- (1.25,-2.5);
\node[draw=none] at (2.15,-1.75) {$\cdots$};
\draw[very thick] (-.15,-1.5) arc (180:360:.9 and 0.15);
\begin{pgfonlayer}{bg}
\draw[very thick] (1.65,-1.5) arc (0:30:.9 and 0.15);
\draw[very thick] (-.035,-1.43) arc (150:180:.9 and 0.15);
\end{pgfonlayer}
\node[draw=none] at (-.55,-1.5) {$\frac{1}{n}$};
\end{tikzpicture}
\caption{$K(p,q,r,-n)$}
\label{figure: og braid}
\end{center}
\end{figure}

\begin{figure}
\begin{center}
\begin{tikzpicture}
\draw (0,0) -- (0,1);
\draw (.25,0) -- (.25,1);
\draw (.5,0) -- (.5,1);
\node[draw=none] at (.9,.5) {$\cdots$};
\draw (1.25, 0) -- (1.25,1);
\draw (1.5, 0) -- (1.5,1);
\draw (1.75, 0) -- (1.75,1);
\node[draw=none] at (.9,.5) {$\cdots$};
\node[draw=none] at (2.15,.5) {$\cdots$};
\draw (2.5,0) -- (2.5,1);
\draw (2.75,0) --(2.75,1);
\node[draw=none] at (1.4,-.5) {$e$};
\draw (-.1, -1) rectangle (2.85, 0); 
\draw (0, -1) -- (0,-1.5);
\draw (0.25, -1) -- (0.25,-1.5);
\draw (.5,-1) -- (.5,-1.5);
\node[draw=none] at (.9,-1.25) {$\cdots$};
\draw (1.25, -1) -- (1.25,-1.5);
\draw (1.5,-1) -- (1.5,-1.5);
\draw (1.75,-1) --(1.75,-1.5);
\draw (2.5,-1) -- (2.5,-1.5);
\draw (2.75,-1) --(2.75,-1.5);
\draw[very thick] (-.175, -1.5) arc (180:360:1.55 and 0.1);
\begin{pgfonlayer}{bg}
\draw[very thick] (2.94,-1.5) arc (0:30:.9 and 0.12);
\draw[very thick] (-.03,-1.43) arc (150:180:.9 and 0.15);
\end{pgfonlayer}
\node[draw=none] at (-.55,-1.5) {$-\frac{1}{k}$};

\node[draw=none] at (.9,-1.86) {$\cdots$};
\node[draw=none] at (2.15,-1.25) {$\cdots$};
\draw (0, -1.65) -- (0,-2);
\draw (0.25, -1.7) -- (0.25,-2);
\draw (.5,-1.72) -- (.5,-2);
\draw (1.25, -1.72) -- (1.25,-2);
\draw (1.5,-1.72) -- (1.5,-2);
\draw (1.75,-1.72) --(1.75,-2.5);
\draw (2.5,-1.7) -- (2.5,-2.5);
\draw (2.75,-1.65) --(2.75,-2.5);
\draw[very thick] (-.15,-2) arc (180:360:.9 and 0.1);
\begin{pgfonlayer}{bg}
\draw[very thick] (1.65,-2) arc (0:30:.9 and 0.12);
\draw[very thick] (-.035,-1.95) arc (150:180:.9 and 0.12);
\end{pgfonlayer}
\node[draw=none] at (-.65,-2) {$\frac{1}{n}$};
\draw (0, -2.2) -- (0,-2.5);
\draw (0.25, -2.2) -- (0.25,-2.5);
\draw (.5, -2.2) -- (.5,-2.5);
\node[draw=none] at (.9,-2.35) {$\cdots$};
\draw (1.25, -2.2) -- (1.25, -2.5);
\draw (1.5, -2.2) -- (1.5, -2.5);
\node[draw=none] at (2.15,-2.075) {$\cdots$};
\end{tikzpicture}
\caption{$K(p,q,r,-n)$}
\label{figure: all link surgery}
\end{center}
\end{figure}

\begin{figure}
\begin{center}
\begin{tikzpicture}
\draw (0,0) -- (0,1);
\draw (.25,0) -- (.25,1);
\draw (.5,0) -- (.5,1);
\node[draw=none] at (.9,.5) {$\cdots$};
\draw (1.25, 0) -- (1.25,1);
\draw (1.5, 0) -- (1.5,1);
\draw (1.75, 0) -- (1.75,1);
\node[draw=none] at (.9,.5) {$\cdots$};
\node[draw=none] at (2.15,.5) {$\cdots$};
\draw (2.5,0) -- (2.5,1);
\draw (2.75,0) --(2.75,1);
\node[draw=none] at (1.4,-.5) {$e$};
\draw (-.1, -1) rectangle (2.85, 0); 

\node[draw=none] at (2.15,-1.25) {$\cdots$};
\draw plot [smooth, tension=1] coordinates { (0, -1) (.05, -1.75) (.25, -2.3)};
\draw plot [smooth, tension=1] coordinates { (0.25, -1) (.3, -1.75) (.5, -2.2)};
\draw plot [smooth, tension=1] coordinates { (.5, -1) (.55, -1.75) (.75, -2.1)};
\node[draw=none] at (.9,-1.25) {$\cdots$};
\draw plot [smooth, tension=1] coordinates { (1.25, -1) (1.3, -1.65) (1.45, -1.9)};
\draw plot [smooth, tension=1] coordinates { (1.5, -1) (1.51, -1.5) (1.55, -1.7)};
\draw plot [smooth, tension=.4] coordinates { (1.75,-1) (1.5,-2) (.25,-2.5)  (0,-3)(.6, -3.55)};
\node[draw=none] at (.5,-2.75) {$\cdots$};
\draw plot [smooth, tension=.4] coordinates { (2.5,-1) (2.25,-2) (1,-2.5)  (.75,-3) (.95,-3.2)};
\draw plot [smooth, tension=.4] coordinates { (2.75,-1) (2.5,-2) (1.25,-2.5)  (1,-3)(1.05, -3.1)};
\draw plot [smooth, tension=.5] coordinates {(1.15, -2.8) (1.25,-3)(.1,-4.3) (0,-4.5)};
\draw plot [smooth, tension=.5] coordinates {(1.25, -2.55) (1.5,-3)(.35,-4.3)(.25,-4.5)};
\draw plot [smooth, tension=.5] coordinates {(1.5, -2.5) (1.75,-3)(.6, -4.3)(.5,-4.5)};
\node[draw=none] at (2.15,-2.75) {$\cdots$};
\draw plot [smooth, tension=.5] coordinates {(2.25, -2.25)(2.5,-3)(1.35,-4.3)(1.25,-4.5)};
\draw plot [smooth, tension=.5] coordinates {(2.5, -2.1) (2.75,-3)(1.6,-4.3)(1.5, -4.5)};
\draw plot [smooth, tension=.5] coordinates {(1.7, -4.35) (1.8,-4.4) (1.85, -4.5)};
\node[draw=none] at (2.15,-4.25) {$\cdots$};
\draw plot [smooth, tension=1] coordinates {(2.35, -3.75) (2.425,-4) (2.5, -4.5)};
\draw plot [smooth, tension=1] coordinates {(2.5, -3.5) (2.625,-3.9) (2.75, -4.5)};
\begin{pgfonlayer}{bg}
\draw[very thick] (-.15,-4.5) arc (180:360:.9 and 0.1);
\draw[very thick] (1.65,-4.5) arc (0:30:.9 and 0.12);
\draw[very thick] (-.035,-4.45) arc (150:180:.9 and 0.12);
\node[draw=none] at (-.5,-4.5) {$-1$};
\draw[very thick] (1.75,-4.5) arc (180:360:.55 and 0.1);
\draw[very thick] (2.85,-4.5) arc (0:30:.55 and 0.12);
\draw[very thick] (1.8,-4.45) arc (150:180:.25 and 0.12);
\end{pgfonlayer}
\node[draw=none] at (3.25,-4.5) {$-1$};
\draw (0,-4.65) -- (0, -4.85);
\draw (.25,-4.65) -- (.25, -4.85);
\draw (.25,-4.65) -- (.25, -4.85);
\draw (.5,-4.65) -- (.5, -4.85);
\draw (1.25,-4.65) -- (1.25, -4.85);
\draw (1.5,-4.65) -- (1.5, -4.85);
\draw (1.85,-4.65) -- (1.85, -4.85);
\draw (2.5,-4.65) -- (2.5, -4.85);
\draw (2.75,-4.65) -- (2.75, -4.85);
\node[draw=none] at (1.5,-5.1) {$\vdots$};

\draw plot [smooth, tension=1] coordinates { (0, -5.5) (.05, -6.25) (.25, -6.8)};
\draw plot [smooth, tension=1] coordinates { (0.25, -5.5) (.3, -6.25) (.5, -6.7)};
\draw plot [smooth, tension=1] coordinates { (.5, -5.5) (.55, -6.25) (.75, -6.6)};
\node[draw=none] at (.9,-5.75) {$\cdots$};
\draw plot [smooth, tension=1] coordinates { (1.25, -5.5) (1.3, -6.15) (1.45, -6.4)};
\draw plot [smooth, tension=1] coordinates { (1.5, -5.5) (1.51, -6) (1.55, -6.2)};
\draw plot [smooth, tension=.4] coordinates { (1.75,-5.5) (1.5,-6.5) (.25,-7)  (0,-7.5)(.6, -8.05)};
\node[draw=none] at (.5,-7.25) {$\cdots$};
\draw plot [smooth, tension=.4] coordinates { (2.5,-5.5) (2.25,-6.5) (1,-7)  (.75,-7.5) (.95,-7.7)};
\draw plot [smooth, tension=.4] coordinates { (2.75,-5.5) (2.5,-6.5) (1.25,-7)  (1,-7.5)(1.05, -7.6)};
\draw plot [smooth, tension=.5] coordinates {(1.15, -7.3) (1.25,-7.5)(.1,-8.8) (0,-9)};
\draw plot [smooth, tension=.5] coordinates {(1.25, -7.05) (1.5,-7.5)(.35,-8.8)(.25,-9)};
\draw plot [smooth, tension=.5] coordinates {(1.5, -7) (1.75,-7.5)(.6, -8.8)(.5,-9)};
\node[draw=none] at (2.15,-7.25) {$\cdots$};
\draw plot [smooth, tension=.5] coordinates {(2.25, -6.75)(2.5,-7.5)(1.35,-8.8)(1.25,-9)};
\draw plot [smooth, tension=.5] coordinates {(2.5, -6.6) (2.75,-7.5)(1.6,-8.8)(1.5, -9)};
\draw plot [smooth, tension=.5] coordinates {(1.7, -8.85) (1.8,-8.9) (1.85, -9)};
\node[draw=none] at (2.15,-8.75) {$\cdots$};
\draw plot [smooth, tension=1] coordinates {(2.35, -8.25) (2.425,-8.5) (2.5, -9)};
\draw plot [smooth, tension=1] coordinates {(2.5, -8) (2.625,-8.4) (2.75, -9)};
\begin{pgfonlayer}{bg}
\draw[very thick] (-.15,-9) arc (180:360:.9 and 0.1);
\draw[very thick] (1.65,-9) arc (0:30:.9 and 0.12);
\draw[very thick] (-.035,-8.95) arc (150:180:.9 and 0.12);
\node[draw=none] at (-.5,-9) {$-1$};
\draw[very thick] (1.75,-9) arc (180:360:.55 and 0.1);
\draw[very thick] (2.85,-9) arc (0:30:.55 and 0.12);
\draw[very thick] (1.8,-8.95) arc (150:180:.25 and 0.12);
\end{pgfonlayer}
\node[draw=none] at (3.25,-9) {$-1$};
\draw (0,-9.15) -- (0, -9.35);
\draw (.25,-9.15) -- (.25, -9.35);
\draw (0.5,-9.15) -- (.5, -9.35);
\draw (1.25,-9.15) -- (1.25, -9.35);
\draw (1.5,-9.15) -- (1.5, -9.35);
\draw (1.85,-9.15) -- (1.85, -9.75);
\draw (2.5,-9.15) -- (2.5, -9.75);
\draw (2.75,-9.15) -- (2.75, -9.75);
\begin{pgfonlayer}{bg}
\draw[very thick] (-.15,-9.35) arc (180:360:.9 and 0.1);
\draw[very thick] (1.65,-9.35) arc (0:30:.9 and 0.12);
\draw[very thick] (-.035,-9.3) arc (150:180:.9 and 0.12);
\end{pgfonlayer}
\node[draw=none] at (-.55,-9.4) {$\frac{1}{n}$};
\draw (0,-9.55) -- (0, -9.75);
\draw (.25,-9.55) -- (.25, -9.75);
\draw (0.5,-9.55) -- (.5, -9.75);
\draw (1.25,-9.55) -- (1.25, -9.75);
\draw (1.5,-9.55) -- (1.5, -9.75);
\end{tikzpicture}
\caption{$K(p,q,r,-n)$}
\label{figure: all twists}
\end{center}
\end{figure}

\begin{figure}
\begin{center}
\begin{tikzpicture}
\draw (0,2) -- (0,1);
\draw (.25,2) -- (.25,1);
\draw (.5,2) -- (.5,1);
\draw (1.25, 2) -- (1.25,1);
\draw (1.5, 2) -- (1.5,1);
\draw (1.75, 2) -- (1.75,1);
\node[draw=none] at (.9,1.5) {$\cdots$};
\node[draw=none] at (2.15,1.5) {$\cdots$};
\draw (2.5,2) -- (2.5,1);
\draw (2.75,2) --(2.75,1);
\node[draw=none] at (1.4,.5) {$e$};
\draw (-.1, 0) rectangle (2.85, 1); 
\draw (0,0) -- (0,-.25);
\draw (.25,0) -- (.25,-.25);
\draw (.5,0) -- (.5,-.25);
\node[draw=none] at (.9,-.125) {$\cdots$};
\draw (1.25, 0) -- (1.25,-.25);
\draw (1.5, 0) -- (1.5,-.25);
\draw[very thick] (-.075,-.25) arc (180:360:.85 and 0.1);
\draw[very thick] (1.6,-.25) arc (0:30:.65 and 0.12);
\draw[very thick] (0.0,-.2) arc (150:180:.55 and 0.12);
\node[draw=none] at (-.75,-.25) {$\frac{1}{n-k}$};
\draw  (1.75,0)--(1.75,-.75);
\node[draw=none] at (2.15,-.125) {$\cdots$};
\draw (2.5,0)  -- (2.5,-.75);
\draw  (2.75,0)--(2.75,-.75);
\draw [dashed] (-.5, -.5) -- (3.25, -.5);
\begin{pgfonlayer}{bg}
\draw[very thick] (1.675,-.75) arc (180:360:.575 and 0.09);
\draw[very thick] (2.825,-.75) arc (0:30:.3 and 0.12);
\draw[very thick] (1.725,-.7) arc (150:180:.3 and 0.12);
\node[draw=none] at (3.25,-.75) {$-\frac{1}{k}$};
\end{pgfonlayer}
\draw (0,-.4) -- (0,-1);
\draw (.25,-.4) -- (.25,-1);
\draw (.5,-.4) -- (.5,-1);
\draw (1.25, -.4) -- (1.25,-1);
\draw (1.5, -.4) -- (1.5,-1);
\draw  (1.75,-.9)--(1.75,-1);
\draw (2.5,-.9)  -- (2.5,-1);
\draw  (2.75,-.9)--(2.75,-1);
\draw plot [smooth, tension=1] coordinates { (0, -1) (.05, -1.75) (.25, -2.3)};
\draw plot [smooth, tension=1] coordinates { (0.25, -1) (.3, -1.75) (.5, -2.2)};
\draw plot [smooth, tension=1] coordinates { (.5, -1) (.55, -1.75) (.75, -2.1)};
\node[draw=none] at (.9,-1.25) {$\cdots$};
\draw plot [smooth, tension=1] coordinates { (1.25, -1) (1.3, -1.65) (1.45, -1.9)};
\draw plot [smooth, tension=1] coordinates { (1.5, -1) (1.51, -1.5) (1.55, -1.7)};
\draw plot [smooth, tension=.4] coordinates { (1.75,-1) (1.5,-2) (.25,-2.5)  (0,-3)(.6, -3.55)};
\node[draw=none] at (.5,-2.75) {$\cdots$};
\draw plot [smooth, tension=.4] coordinates { (2.5,-1) (2.25,-2) (1,-2.5)  (.75,-3) (.95,-3.2)};
\draw plot [smooth, tension=.4] coordinates { (2.75,-1) (2.5,-2) (1.25,-2.5)  (1,-3)(1.05, -3.1)};
\draw plot [smooth, tension=.5] coordinates {(1.15, -2.8) (1.25,-3)(.1,-4.3) (0,-4.5)};
\draw plot [smooth, tension=.5] coordinates {(1.25, -2.55) (1.5,-3)(.35,-4.3)(.25,-4.5)};
\draw plot [smooth, tension=.5] coordinates {(1.5, -2.5) (1.75,-3)(.6, -4.3)(.5,-4.5)};
\node[draw=none] at (2.15,-2.75) {$\cdots$};
\draw plot [smooth, tension=.5] coordinates {(2.25, -2.25)(2.5,-3)(1.35,-4.3)(1.25,-4.5)};
\draw plot [smooth, tension=.5] coordinates {(2.5, -2.1) (2.75,-3)(1.6,-4.3)(1.5, -4.5)};
\draw plot [smooth, tension=.5] coordinates {(1.7, -4.35) (1.8,-4.4) (1.85, -4.5)};
\node[draw=none] at (2.15,-4.25) {$\cdots$};
\draw plot [smooth, tension=1] coordinates {(2.35, -3.75) (2.425,-4) (2.5, -4.5)};
\draw plot [smooth, tension=1] coordinates {(2.5, -3.5) (2.625,-3.9) (2.75, -4.5)};

\node[draw=none] at (1.5,-5) {$\vdots$};

\draw plot [smooth, tension=1] coordinates { (0, -5.5) (.05, -6.25) (.25, -6.8)};
\draw plot [smooth, tension=1] coordinates { (0.25, -5.5) (.3, -6.25) (.5, -6.7)};
\draw plot [smooth, tension=1] coordinates { (.5, -5.5) (.55, -6.25) (.75, -6.6)};
\node[draw=none] at (.9,-5.75) {$\cdots$};
\draw plot [smooth, tension=1] coordinates { (1.25, -5.5) (1.3, -6.15) (1.45, -6.4)};
\draw plot [smooth, tension=1] coordinates { (1.5, -5.5) (1.51, -6) (1.55, -6.2)};
\draw plot [smooth, tension=.4] coordinates { (1.75,-5.5) (1.5,-6.5) (.25,-7)  (0,-7.5)(.6, -8.05)};
\node[draw=none] at (.5,-7.25) {$\cdots$};
\draw plot [smooth, tension=.4] coordinates { (2.5,-5.5) (2.25,-6.5) (1,-7)  (.75,-7.5) (.95,-7.7)};
\draw plot [smooth, tension=.4] coordinates { (2.75,-5.5) (2.5,-6.5) (1.25,-7)  (1,-7.5)(1.05, -7.6)};
\draw plot [smooth, tension=.5] coordinates {(1.15, -7.3) (1.25,-7.5)(.1,-8.8) (0,-9)};
\draw plot [smooth, tension=.5] coordinates {(1.25, -7.05) (1.5,-7.5)(.35,-8.8)(.25,-9)};
\draw plot [smooth, tension=.5] coordinates {(1.5, -7) (1.75,-7.5)(.6, -8.8)(.5,-9)};
\node[draw=none] at (2.15,-7.25) {$\cdots$};
\draw plot [smooth, tension=.5] coordinates {(2.25, -6.75)(2.5,-7.5)(1.35,-8.8)(1.25,-9)};
\draw plot [smooth, tension=.5] coordinates {(2.5, -6.6) (2.75,-7.5)(1.6,-8.8)(1.5, -9)};
\draw plot [smooth, tension=.5] coordinates {(1.7, -8.85) (1.8,-8.9) (1.85, -9)};
\node[draw=none] at (2.15,-8.75) {$\cdots$};
\draw plot [smooth, tension=1] coordinates {(2.35, -8.25) (2.425,-8.5) (2.5, -9)};
\draw plot [smooth, tension=1] coordinates {(2.5, -8) (2.625,-8.4) (2.75, -9)};
\end{tikzpicture}
\caption{$K(p,q,r,-n)$}
\label{figure: surgery pushed up}
\end{center}
\end{figure}

\begin{figure}
\begin{center}
\begin{tikzpicture}
\draw (0,2) -- (0,1);
\draw (.25,2) -- (.25,1);
\draw (.5,2) -- (.5,1);
\draw (1.25, 2) -- (1.25,1);
\draw (1.5, 2) -- (1.5,1);
\draw (1.75, 2) -- (1.75,1);
\node[draw=none] at (.9,1.5) {$\cdots$};
\node[draw=none] at (2.15,1.5) {$\cdots$};
\draw (2.5,2) -- (2.5,1);
\draw (2.75,2) --(2.75,1);
\node[draw=none] at (1.4,.5) {$e-r$};
\draw (-.1, 0) rectangle (2.85, 1); 
\draw plot [smooth, tension=.5] coordinates {(1.5,0) (1.25, -.25) (.1, -.4) (0, -1)};
\draw plot [smooth, tension=.5] coordinates {(1.75,0) (1.55, -.4) (.35, -.6) (0.25, -1)};
\node[draw=none] at (2.15,-.125) {$\cdots$};
\draw plot [smooth, tension=.5] coordinates {(2.5,0) (2.3, -.5) (1.6, -.7) (1.5, -1)};
\draw plot [smooth, tension=.5] coordinates {(2.75,0) (2.6, -.5) (1.85, -.8) (1.75, -1)};
\draw (0,0) -- (.1,-.3); \draw (.15, -.45) -- (.2, -.6);
\node[draw=none] at (.7,-.125) {$\cdots$};
\draw (1.25, 0) -- (1.3,-.15); \draw (1.35, -.25)--(1.383,-.35);
\draw plot [smooth, tension = 1] coordinates{(1.93,-.8)(2,-1)(2,-1.4)};
\node[draw=none] at (2.385,-.85) {$\cdots$};
\draw plot [smooth, tension=1] coordinates {(2.6,-.55) (2.75, -1) (2.75,-1.4)}; 
\draw[very thick] (-.125,-1) arc (180:360:.99 and 0.1);
\draw[very thick] (1.85,-1) arc (0:30:.45 and 0.12);
\draw[very thick] (-.05,-.95) arc (150:180:.55 and 0.12);
\draw (0,-1.15)--(0,-1.4);
\draw (0.25,-1.15)--(0.25,-1.4);
\node[draw=none] at (.875,-1.25) {$\cdots$};
\draw (1.5,-1.15)--(1.5,-1.4);
\draw (1.75,-1.15)--(1.75,-1.4);
\node[draw=none] at (-1.25,-1) {$\frac{1}{n-k-1}$};
\node[draw=none] at (.875,-1.6) {$\underbrace{\ \ \ \ \ \ \ \ \ \ \ \ \ \ \ \ \ }$};
\node[draw=none] at (.875,-1.85) {$r$ strands};
\end{tikzpicture}
\caption{Top of $K(p,q,r,-n)$}
\label{figure: top of braid}
\end{center}
\end{figure}
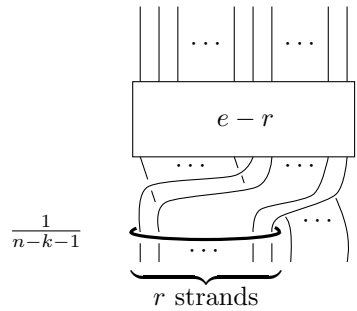

\section{Inequivalent positions of knots}
\label{section:InequivalentPositionsOfKnots}

In this section, we will discuss knots positioned on a genus two Heegaard surface, and relationships between such a knot and the handlebodies on either side. We will also introduce the \emph{extended Goeritz group} and discuss its induced action on the homology of a surface.

\subsection{Primitive and Seifert positions}
\label{subsection:PrimitiveSeifertPositions}

We begin by introducing the notion of Dehn surgery. Let $K$ be a knot in $S^3$, and let $N(K)$ denote a closed tubular neighborhood of $K$. A \emph{slope} on $K$ is an isotopy class of curves in $\bd N(K)$.

Let $\alpha$ be a slope on $K$. Let $M_K$ be the manifold obtained by removing the interior of $N(K)$ from $S^3$, so that $M_K$ is a three-manifold with boundary, and that boundary is a torus, $\bd N(K)$. Let $M_K(\alpha)$ be the quotient three-manifold obtained by attaching a solid torus, $D^2 \times S^1$, to $M_K$ so that the slope $\alpha$ is identified with a non-trivial isotopy class of curves in $\bd D^2 \times S^1$ that bound a disk in $D^2 \times S^1$. This process is called \emph{Dehn surgery} on $K$ along $\alpha$.

Dehn surgery is an important and long-studied area of low-dimensional topology, motivated initially by the incredible fact that every possible closed, compact, connected, orientable three-manifold can be obtained by some finite collection of simultaneous Dehn surgeries on the components of some link in $S^3$ (see \cite{WalMCM}). 

\begin{definition} 
Let $K$ be a simple closed curve embedded in a closed surface $F$ in $S^3$. Let $N(K)$ be a closed tubular neighborhood of $K$ in $S^3$. The \emph{surface slope} of $K$ induced by $F$ is the isotopy class in $\bd N(K)$ of a component of $F \cap \bd N(K)$.
\end{definition}

When the surface $F$ is understood, the reference to $F$ is often omitted. For instance, the twisted torus knots we have discussed lie on a Heegaard surface of genus two, and therefore have a naturally associated slope induced by this Heegaard surface.

John Berge recognized that if a knot sits on a Heegaard surface in such a way that it has a nice position in relationship to both of the handlebodies on either side of the surface, then Dehn surgery on the knot along the surface slope would yield a predictable manifold. This precipitated a great deal of research into which positions were interesting, and what the resulting manifolds would be. 

Let $H$ be a three-manifold with boundary, and let $K$ be a simple closed curve embedded in $\bd H$. A \emph{2-handle attachment} to $H$ along $K$ is the quotient three-manifold obtained by attaching a 2-handle, $D^2 \times I$, to $H$ so that the annular region, $\bd D^2 \times I$, is identified with an annular neighborhood of $K$ in $\bd H$. The resulting manifold is denoted $H[K]$.

If $H$ is a genus two handlebody, and $K$ is a simple closed curve embedded in $\bd H$, we say that $K$ is \emph{primitive} with respect to $H$ if $H[K]$ is a solid torus, or we say that $K$ is \emph{Seifert} with respect to $H$ if $H[K]$ is a manifold in a special class, called a Seifert fibered space.

\begin{definition} Let $K$ be a knot sitting on a genus two Heegaard surface bounding two handlebodies, $H$ and $H'$. Then $K$ is called \emph{primitive/primitive} (or just \emph{p/p}) if $K$ is simultaneously primitive with respect to $H$ and $H'$, or $K$ is called \emph{primitive/Seifert} (or just \emph{p/S}) if $K$ is simultaneously primitive with respect to $H$ and Seifert with respect to $H'$ (or vice versa).
\end{definition}

Among others, Berge, Dean \cite{DeaSSFDSHK}, and Eudave-Mu\~{n}oz \cite{EudHKSFDS} have studied p/p, and p/S knots.

We will make use of some important results of Dean. When Dean introduced the twisted torus knots, he included an additional (unnecessary) parameter, and deduced a method of determining whether a twisted torus knot is primitive or Seifert with respect to either handlebody by calculating an algebraic term $w_{p, q, r, m, n}$ in terms of these parameters. We will adapt this term for use with our notation.

\begin{definition}
If the twisted torus knot $K = K(p, q, r, n)$ sits on a genus two Heegaard surface bounding handlebodies $H$ and $H'$, then the conjugacy class of $K$ in the fundamental group of one of the handlebodies is denoted $w_{p, q, r, n, 1}$, and the in the other as $w'_{q, p, r, 1, n}$. 
\end{definition}

Without delving too deeply into these terms, we will use the following important facts.

\begin{lemma}[Lemma 3.3 of \cite{DeaSSFDSHK}]
\label{lemma:DeanLemma}
The term $w_{p, q, r, m, n}$ satisfies:
\begin{enumerate}
\item[(i)] $w_{p, q, r, m, n}$ is equivalent to $w_{p, q', r, m, n}$ if $q \equiv \pm q'$ mod $p$.
\item[(ii)] $w_{p, q, r, m, n}$ is equivalent to $w_{p, q, r', m, n}$ if $r \equiv \pm r'$ mod $p$.
\end{enumerate}
\end{lemma}

\begin{theorem}[Theorem 3.4 of \cite{DeaSSFDSHK}]
\label{theorem:DeanTheorem}
The torus knot $K(p, q, r, n)$ is primitive with respect to one of the handlebodies if and only if
\begin{enumerate}
\item[(i)] $p=1$; or
\item[(ii)] $n=\pm1$ and $r \equiv \pm 1$ or $\pm q$ mod $p$.
\end{enumerate}
\end{theorem}

We remark that the original statement of Theorem 3.4 in \cite{DeaSSFDSHK} requires that $n = 1$, but this is the consequence of an assumption to simplify the exposition, and it is clear from the proof that it applies just as well when $n = -1$.

Dean also provides criteria for a twisted torus knot to be Seifert with respect to a particular handlebody. He refers to three cases as \emph{hyper}-, \emph{end}-, or \emph{middle}-Seifert-fibered, depending on the values of the parameters $n$ and $r$. He conjectures that these describe all cases of twisted torus knots that are Seifert. We will say that a knot is \emph{HEM}-Seifert if it is one of these three cases.

\subsection{Extended Goeritz group}
\label{subsection:ExtendedGoeritzGroup}

In \cite{SchaA3SPG2HS}, Scharlemann proves that the group of isotopy classes of
orientation-preserving homeomorphisms of $S^3$ that preserve a genus 2 Heegaard
splitting (preserving each handlebody set-wise) is generated by four elements: $\alpha$, $\beta$, $\gamma$ and
$\delta$. The automorphisms $\alpha$, $\beta$ and $\gamma$ are as shown in Figs. \ref{fig:alpgam}
and \ref{fig:beta}, where each rotation is by $\pi$. 

\begin{figure}[h]
\begin{center}
\includegraphics[scale=.5]{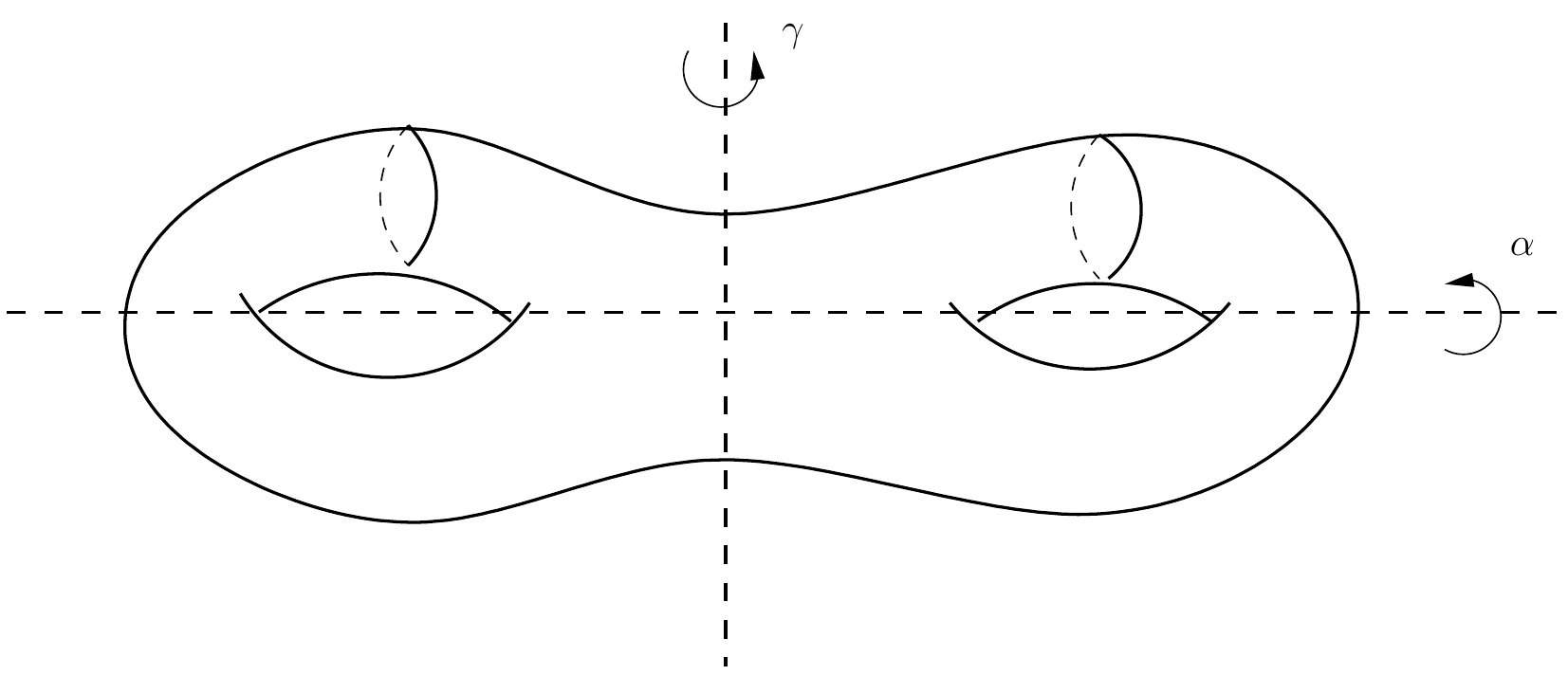} 
\caption{The automorphisms $\alpha$ and $\gamma$ both induce rotations of each entire handlebody by $\pi$.}\label{fig:alpgam}
\end{center}
\end{figure}

\begin{figure}[h!]
 \begin{center}
\includegraphics[trim={2in 6.5in 1in 2in}, scale=.5]{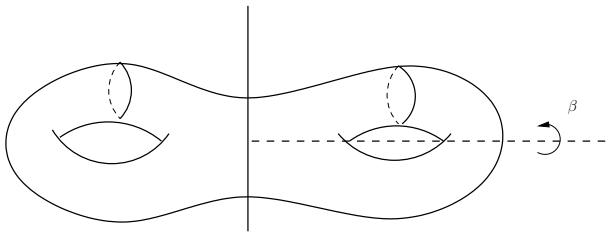} 
\caption{The automorphism $\beta$ acts on one of the handlebodies by a rotation of half of the handlebody by $\pi$.}\label{fig:beta}
\end{center}
\end{figure}

As described in \cite{SchaA3SPG2HS}, the element $\delta$ can be chosen from many
possible operations. Here we think of the genus two handlebody as a 2-sphere with
two 1-handles attached to it. The homeomorphism $\delta$ is the one that slides
the foot of one of the 1-handles over a longitudinal curve of the other and back
to its original position. 

These automorphisms preserve $H$ and $H'$, individually. To fully describe orientation-preserving automorphisms of $S^3$ that preserve the Heegaard surface, we must also allow for automorphisms that exchange $H$ and $H'$. Let $\varepsilon$ be obtained from rotation by $\pi$ about the curve $C$, as
shown in Fig. \ref{fig:eps}. 

\begin{figure}[h]
 \begin{center}
\includegraphics[scale=.5]{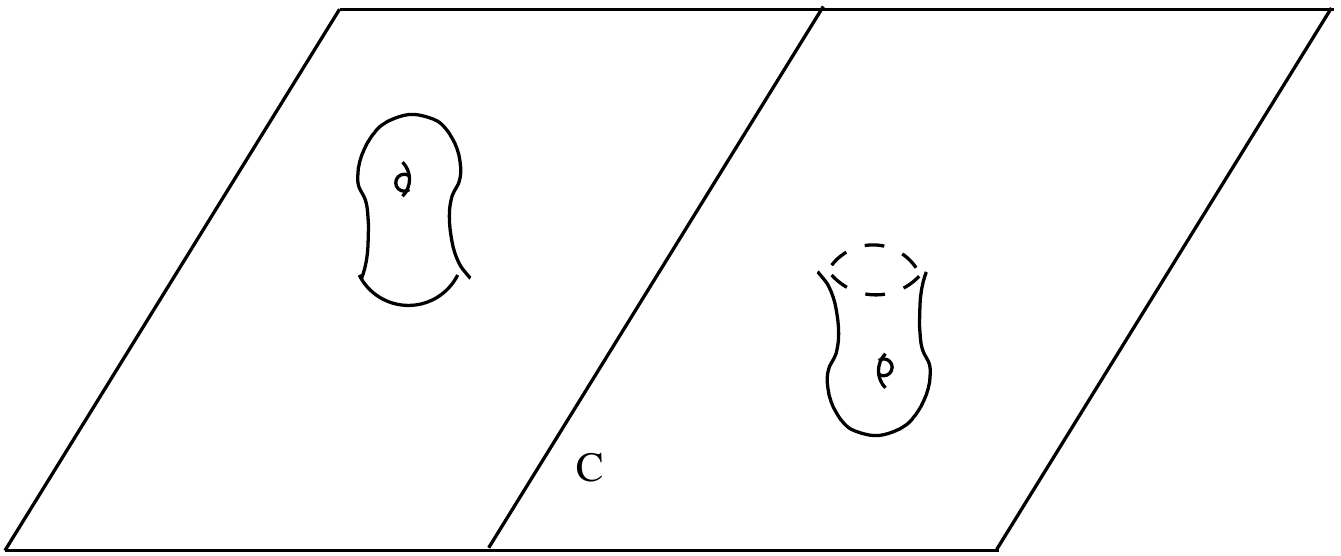} 
\caption{The automorphism $\varepsilon$ is orientation-preserving, and exchanges the two handlebodies $H$ and $H'$.}
\label{fig:eps}
\end{center}
\end{figure}

We will call the group generated by $\alpha$, $\beta$, $\gamma$, $\delta$ and
$\varepsilon$ the \textit{extended Goeritz group}. Observe that this includes every isotopy class of orientation-preserving homeomorphism of $S^3$ that preserves a genus two Heegaard \emph{surface}.

We will be considering the induced action of the extended Goeritz group on the homology of the Heegaard surface. First, label the
generators of $H_1( F)$ as $a$, $x$, $b$, and $y$, as shown in Fig. \ref{fig:coords}. We will adopt the convention of representing elements of $H_1(F)$ in the form $(a, x, b, y)$, listing both meridians before their duals.

\begin{figure}[h]
\begin{center}
\includegraphics[scale=.6]{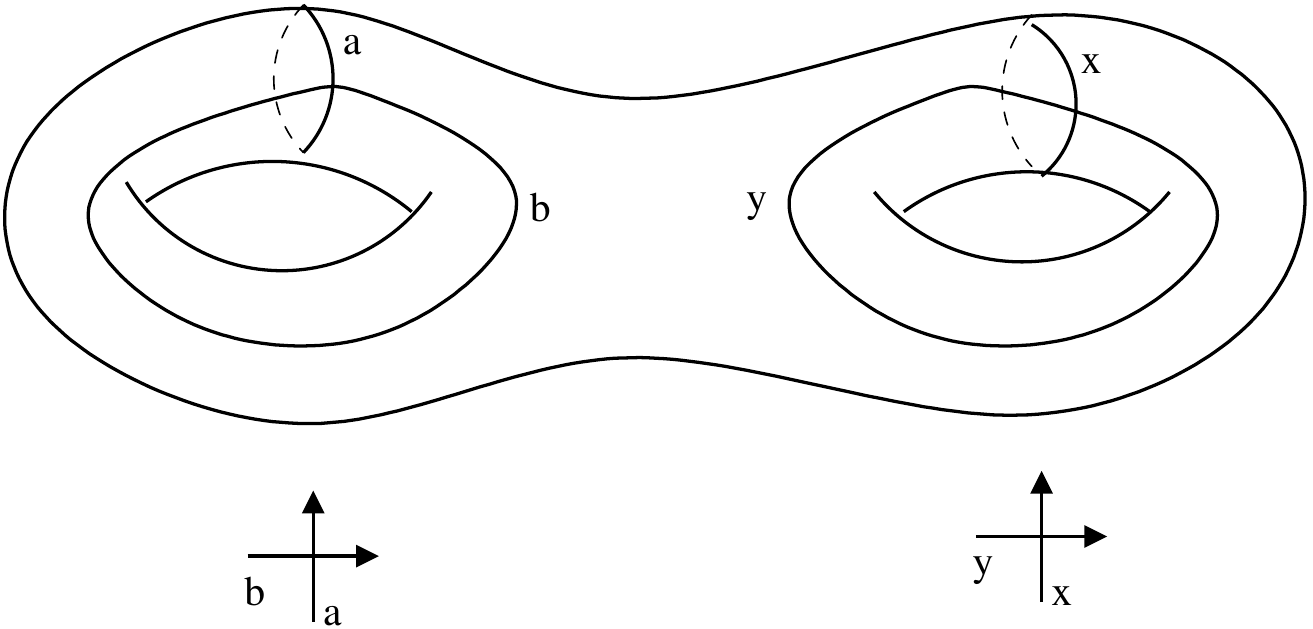} 
\caption{Generators of $H_1( F) = \langle a \rangle \oplus \langle x \rangle \oplus \langle b \rangle \oplus \langle y \rangle$.}\label{fig:coords}
\end{center}
\end{figure}

\begin{lemma}
\label{lemma:TwistedTorusKnotInH1}
In $H_1(F)$, the twisted torus knot $K(p, q, r, n)$ is represented by the class $[K(p, q, r, n)] = (q, nr, -p, -r)$.
\end{lemma}

\begin{proof}
In $H_1(F)$, the $(p, q)$-torus knot lying on the side of $F$ with generators $a$ and $b$ represents $(q, 0, -p, 0)$, and the $(1, n)$-torus knot represents $(0, n, 0, -1)$. Observe that the twisted torus knot can be defined correctly from these two torus knots with these orientations, and that the $(p,q)$-torus knot has positive crossings, while the $(1, n)$-torus knot will have crossings that agree with the sign of $n$. Then, as the twisted torus knot $K(p, q, r, n)$ has $r$ strands twisting $n$ times forming a $r$ parallel $(1, n)$-torus knot, the result follows.
\end{proof}

Denote by $\as$, $\bs$, $\gs$, $\ds$, and $\es$ the maps induced on $H_1(
F)$ by $\alpha$, $\beta$, $\gamma$, $\delta$, and $\varepsilon$, respectively.
Considering elements of $H_1( F)$ in the form $( a, x, b, y)$, the induced maps act in the following way:

\[
\begin{array}{lllll}
\as : & a \mapsto -a,  & x \mapsto -x, & b \mapsto -b, & y \mapsto -y \\
\bs: & a \mapsto a, & x \mapsto -x, & b \mapsto b, & y \mapsto -y \\
\gs: &  a \mapsto -x, & x \mapsto -a, & b \mapsto -y, & y \mapsto -b \\
\ds: & a \mapsto a+x, & x \mapsto x, & b \mapsto b, & y \mapsto y-b \\
\es: &  a \mapsto y, & x \mapsto b, & b \mapsto x, & y \mapsto a.
\end{array}
\]

Using these
definitions, we can define relations on the induced maps. Because $\as$ scales an element of $H_1(F)$ by $-1$, 
 the map $\as$ commutes with all the other induced maps. The other relations are defined as
follows:

\begin{eqnarray*}
\gs \bs & = & \as \bs \gs \\
\es \bs & = & \as \bs \es \\
\ds \bs & = & \bs \ds^{-1} \\
\gs \es & = & \es \gs \\
 \ds \es & = & \es \ds^{-1} .
\end{eqnarray*}

Since $\ds$ and $\gs$ almost commute with $\bs$ and $\es$, and because $\as$, $\bs$, $\gs$, and $\es$ all have order two, we can write the
general form of a composition of these maps as \[ \as^h \bs^j \es^k \gs^l
\ds^{m_1} \gs \ds^{m_2} \gs \cdots \gs \ds^{m_n},\] where $h, j, k, l \in \{ 0, 1
\}$, $m_i \in \mathbb{Z} - \set{0}$ for $i = 1, \ldots, n-1$ and $m_n \in \mathbb{Z}$.

If we consider the $a$, $x$, $b$ and $y$ to be the elementary vectors in $\mathbb R^4$, the maps $\as$, $\bs$, $\gs$ and $\ds$ can be represented as $4\times4$ matrices
over $\mathbb{Z}$, all of the form 
$\begin{pmatrix} 
  A & 0\\ 
  0 & B 
\end{pmatrix}$.
For $\gs$, $A=B= \begin{pmatrix} 
  0 & -1\\ 
  -1 & 0 
\end{pmatrix}$ and for $\ds$, $A = \begin{pmatrix} 
  1 & 0\\ 
  1 & 1 
\end{pmatrix}$ and $B= \begin{pmatrix} 
  1 & -1\\ 
  0 & 1 
\end{pmatrix}$. 

\begin{lemma}\label{lemma:FormOfMap}
The map  $\gs^l \ds^{m_1} \gs \ds^{m_2} \gs \cdots \gs \ds^{m_n}$ where $l \in
\{0,1\}$, $m_i \in \mathbb Z -\{0\}$ for $i = 1, \ldots, n-1$ and $m_n \in \mathbb{Z}$ is
represented by a $4\times4$ matrix over $\mathbb{Z}$ of the form
$\begin{pmatrix} 
  C & 0 \\ 
  0 & ( C^T ) ^{-1} 
\end{pmatrix},$ where $C$ is an invertible $2 \times 2$ matrix with determinant $\pm 1$.
\end{lemma}

\begin{proof}
We use induction on the length of the word $\gs^l \ds^{m_1} \gs \ds^{m_2} \gs
\cdots \gs \ds^{m_n}$. Both $\gs$ and $\ds$ are of the correct form. Suppose the
lemma is true when the word has length $k$. Then we need only check that 
post-composition with $\gs$, $\ds$ and $\ds^{-1}$ preserves this form. Suppose that $C = \begin{pmatrix} 
  s & t\\ 
  u & v 
\end{pmatrix}$, where $sv - tu = \pm 1$.

Post-composition with $\gs$ consists of multiplication of the following blocks:  
\[ \begin{pmatrix} 
  0 & -1\\ 
  -1 & 0 
\end{pmatrix} \begin{pmatrix} 
  s & t\\ 
  u & v 
\end{pmatrix} = \begin{pmatrix} 
  -u & -v\\ 
  -s & -t 
\end{pmatrix} \mbox{ and } \begin{pmatrix} 
  0 & -1\\ 
  -1 & 0 
\end{pmatrix} \left( \frac{1}{sv - tu} \begin{pmatrix} 
  v & -u\\ 
  -t & s 
\end{pmatrix} \right) = \frac{1}{sv - tu}\begin{pmatrix} 
  t & -s\\ 
  -v  & u 
\end{pmatrix}. \]

Post-composition with $\ds$ consists of multiplication of the following blocks: 
\[\begin{pmatrix} 
  1 & 0\\ 
  1 & 1 
\end{pmatrix} \begin{pmatrix} 
  s & t\\ 
  u & v 
\end{pmatrix} = \begin{pmatrix} 
  s & t\\ 
  s+u & t + v 
\end{pmatrix} \mbox{ and } \begin{pmatrix} 
  1 & -1\\ 
  0 & 1 
\end{pmatrix}\left( \frac{1}{sv -tu} \begin{pmatrix} 
  v & -u\\ 
  -t & s 
\end{pmatrix} \right) = \frac{1}{sv - tu} \begin{pmatrix} 
  t+v & -u-s\\ 
  -t  & s 
\end{pmatrix}.\] 

And finally, post-composition with  $\ds^{-1}$ consists of: 
\[\begin{pmatrix} 
  1 & 0\\ 
  -1 & 1 
\end{pmatrix} \begin{pmatrix} 
  s & t\\ 
  u & v 
\end{pmatrix} = \begin{pmatrix} 
  s  & t \\ 
  u-s  & v-t
\end{pmatrix} \mbox{ and } \begin{pmatrix} 
  1 & 1\\ 
  0 & 1 
\end{pmatrix}\left( \frac{1}{sv - tu} \begin{pmatrix} 
  v & -u\\ 
  -t & s 
\end{pmatrix} \right) = \frac{1}{sv - tu}\begin{pmatrix} 
  v-t & s-u \\ 
  -t &  s
\end{pmatrix}.\]

In all cases, the second matrix is the inverse of the transpose of the first matrix.
\end{proof}

\subsection{Inequivalent positions of knots}
\label{subsection:InequivalentPositionsKnots}

Doleshal proved the following.

\begin{theorem}[Theorem 4.1 of \cite{DolFPSTTK}]
For integers $k$, $q$, and $m$ with $k \geq 2$, $q > 2$, $1 \leq m < q$, and $(q, m) = 1$, let $K_1$ and $K_2$ be the twisted torus knots $K_1 = K(kq + m, q, m, -1)$ and $K_2 = K(kq + q - m, q, q - m, -1)$, with their canonical embeddings on the genus two Heeagaard surface, $F$, for $S^3$. Then, $K_1$ and $K_2$ are isotopic as knots in $S^3$ and have the same surface slope with respect to $F$, but there is no homeomorphism of $S^3$ sending the pair $(F, K_1)$ to $(F, K_2)$.
\end{theorem}

We can now prove an extension of Doleshal's result.

\begin{theorem}{DistinctPositionsOfKnots}\label{theorem:DistinctPositionsOfKnots}
For integers $k$, $q$, and $m$ with $k \in \set{0, 1}$, $q > 2$, $1 \leq m < q$, and $(q, m) = 1$, let $K_1$ and $K_2$ be the twisted torus knots $K_1 = K(kq + m, q, m, -1)$ and $K_2 = K(kq + q - m, q, q - m, -1)$, with their canonical embeddings on the genus two Heeagaard surface, $F$, for $S^3$. Then, $K_1$ and $K_2$ are isotopic as knots in $S^3$ and have the same surface slope with respect to $F$. 
\begin{enumerate}
\item[(i)]\label{Case1} When $k = 0$, and 
\begin{enumerate}
\item[(a)] $m = 1$, then $K_1$ is p/p, and $K_2$ is primitive with respect to one handlebody but neither primitive nor HEM-Seifert with resepct to the other.
\item[(b)] $m > 1$, then both $K_1$ and $K_2$ are primitive with respect to one handlebody, but neither primitive nor HEM-Seifert with respect to the other.
\end{enumerate}
\item[(ii)]\label{Case2} When $k = 1$, then both $K_1$ and $K_2$ are p/p.
\end{enumerate}
Further, when $k = 0$, there is no homeomorphism of $S^3$ sending the pair $(F, K_1)$ to $(F, K_2)$.
\end{theorem}

We make two important remarks. First, when $m = 1$ and $k > 0$, the knots are torus knots, and when $k = 0$ and $m = 1$, the knots are actually unknotted in $S^3$. Second, when $k = 1$, it is expected that there is, in fact, a homeomorphism between $(F, K_1)$ and $(F, K_2)$.

\begin{proof}
We observe that the method in \cite{DolFPSTTK} to show that the knots $K_1$ and $K_2$ are isotopic in $S^3$ does not depend on the hypothesis that $k \geq 2$. The exact method of proof establishes the statement for the cases when $k =0$ and $k = 1$ as well.

Let $H$ and $H'$ denote the two handlebodies on either side of the Heegaard surface, $F$, upon which $K_i$ naturally sits as a twisted torus knot. Note first that $K_1$ and $K_2$ both have the same surface slope, $kq^2 + qm - m^2$, by Proposition 3.1 of \cite{GunKDPPPSR}.

We first establish the relations between $K_1$ and $K_2$ to the handlebodies on either side of $F$. 

When $k = 0$, we consider separately the cases of $m=1$ and $m > 1$. 

For $m = 1$, $K_1$ represents the word $w_{1, q, 1, -1, 1}$, in Dean's notation, which is primitive in $\pi_1(H)$ by Theorem \ref{theorem:DeanTheorem}. Additionally, in $\pi_1(H')$, the knot represents $w'_{q, 1, 1, 1, -1}$, which is also primitive by Theorem \ref{theorem:DeanTheorem}. 
Now, $K_2$ represents the word $w_{q-1, q, q-1, -1, 1}$ in $\pi_1(H)$, which is equivalent to $w_{q-1, q, 0, -1, 1}$ by Lemma \ref{lemma:DeanLemma}. From \cite{DeaSSFDSHK}, we conclude that $K_2$ is not HEM-Seifert with respect to $H$. Finally, looking at $K_2$ within $\pi_1(H')$, we have $w'_{q, q-1, q-1, 1, -1}$, which is equivalent to $w'_{q, 1, 1, 1, -1}$ by Lemma \ref{lemma:DeanLemma}. 
This is primitive by Theorem \ref{theorem:DeanTheorem}. Thus, when $k = 0$ and $m=1$, the knot represented by $K_1$ is p/p, and $K_2$ is primitive with respect to $H'$ but neither primitive nor HEM-Seifert with respect to $H$.

Next, for $m > 1$, $K_1$ represents in $\pi_1(H)$ the word $w_{m, q, m, -1, 1}$, which is equivalent to the word $w_{m, q, 0, -1, 1}$ by Lemma \ref{lemma:DeanLemma}, and $K_1$ is neither primitive nor HEM-Seifert with respect to $H$. In $\pi_1(H')$, however, the knot represents the word $w'_{q, m, m, 1, -1}$, which is primitive by Theorem \ref{theorem:DeanTheorem}. The knot $K_2$ represents in $\pi_1(H)$ the word $w_{q-m, q, q-m, -1, 1}$, which is equivalent to $w_{q-m, m, 0, -1, 1}$. This is neither primitive nor HEM-Seifert. In $\pi_1(H')$, the word $w'_{q, q-m, q-m, 1, -1}$ is equivalent to $w'_{q, m, m, 1, -1}$, which is primitive by Theorem \ref{theorem:DeanTheorem}. So, when $k = 0$ and $m > 1$, neither $K_1$ nor $K_2$ are p/p, and if either of the configurations are p/S, then they are Seifert in a way different from those prescribed by Dean.

When $k = 1$, $K_1$ represents $w_{q+m, q, m, -1, 1}$ in $\pi_1(H)$, which is equivalent to $w_{q+m, q, -q, -1, 1}$, 
which shows $K_1$ is primitive
. In $\pi_1(H')$, $K_1$ represents $w'_{q, q+m, m, 1, -1}$, which is equivalent to $w'_{q, m, m, 1, -1}$ 
, which is again primitive
.  Now, $K_2$ represents $w_{2q-m, q, q-m, -1, 1}$ in $\pi_1(H)$, which is equivalent to $w_{2q-m, -(q - m), q-m, -1, 1}$
, which is primitive
. And in $\pi_1(H')$, $K_2$ represents $w'_{q, 2q - m, q-m, 1, -1}$, which is equivalent to $w'_{q, -m, -m, 1, -1}$
, showing the knot to be primitive
. Thus, when $k = 1$, $K_1$ and $K_2$ are both p/p knots.

Only in the case that $k=0$ and $m=1$ can we use the same methods as \cite{DolFPSTTK} to prove that the knots are inequivalent. Instead, we introduce a new method using the extended Goeritz group to show that when $k=0$, there can be no homeomorphism carrying $K_1$ to $K_2$. The goal will be to show that there does not exist an element of the extended
Goeritz group that sends $( F, K_1)$ to $( F, K_2)$. To do so, we will
consider $K_1$ and $K_2$ as the elements they represent in $H_1( F)$ and
study how the extended Goeritz group acts on them. 

From Lemma \ref{lemma:TwistedTorusKnotInH1}, $[K_1] = (q, -m, -kq - m, -m)$, and $[K_2] = (q, m-q, m-q-kq, m-q)$. We begin by acting on $[K_1]$ using the map $\gs^\ell \ds^{m_1} \gs \ds^{m_2} \gs \cdots \gs \ds^{m_n}$, where $\ell \in \set{0, 1}$, $|m_i| > 0$ for $i = 1, \dots, n-1$, and $m_n \in \mathbb{Z}$. By Lemma \ref{lemma:FormOfMap}, the map is represented by a matrix of the form $\begin{pmatrix} 
  C & 0\\ 
  0 & (C^T)^{-1} 
\end{pmatrix}$.

Since $\as$, $\bs$, and $\es$ are all of order two, it is clear that $[K_2]$ can be equal to $\as^h \bs^j \es^k \gs^\ell \ds^{m_1} \gs \ds^{m_2} \gs \cdots \gs \ds^{m_n} [K_1]$ if and only if $\gs^\ell \ds^{m_1} \gs \ds^{m_2} \gs \cdots \gs \ds^{m_n} [K_1] = \as^{h'} \bs^{j'} \es^{k'} [K_2]$, where $h'$, $j'$, and $k'$ are in $\set{0, 1}$. Consider $\gs^\ell \ds^{m_1} \gs \ds^{m_2} \gs \cdots \gs \ds^{m_n} [K_1]$.

We use the general form of the matrix representative, with $sv - tu = 1$, of the map to find a vector 
\[ \vect{v} = \begin{pmatrix} s & t & 0 & 0 \\ u & v & 0 & 0 \\ 0 & 0 &  v & - u \\ 0 & 0 & - t &  s \end{pmatrix} \begin{pmatrix} q \\ -m \\ -kq - m \\ -m \end{pmatrix} = \begin{pmatrix} qs - mt \\ qu - mv \\ mu + (-kq - m)v  \\ -ms + (kq + m)t  \end{pmatrix}. \]

We are trying to determine whether $\vect{v}$ can be equal to $\vect{x} = (x_1, x_2, x_3, x_4)^T = \as^{h'} \bs^{j'} \es^{k'} [K_2]$, so we solve a linear system of equations. To do this, set $\vect{v} = \vect{x}$ to obtain a system of linear equations in $s$, $t$, $u$, and $v$, and use the augmented matrix to solve the system in terms of these variables.

\[ \left( \begin{array}{cccc} q & -m & 0 & 0 \\ 0 & 0 & q & -m \\ 0 & 0 &  m & (-kq - m) \\ - m &  (kq + m) & 0 & 0  \end{array} \right| \left. \begin{array}{c} x_1 \\ x_2 \\ x_3 \\ x_4 \end{array} \right). \]

We seek integer solutions, so we row reduce using only integer multiples of rows, obtaining:
\[ \left( \begin{array}{cccc} q-m & kq & 0 & 0 \\ 0 & 0 & q-m & kq \\ 0 & 0 & m & (-kq-m) \\ -m & (kq+m) & 0 & 0 \end{array} \right| \left. \begin{array}{c} x_1 + x_4 \\ x_2 - x_3 \\ x_3 \\ x_4 \end{array} \right).\]

In the case that $k = 0$, this system is 
\[ \left( \begin{array}{cccc} q-m & 0 & 0 & 0 \\ 0 & 0 & q-m & 0 \\ 0 & 0 & m & -m \\ -m & m & 0 & 0 \end{array} \right| \left. \begin{array}{c} x_1 + x_4 \\ x_2 - x_3 \\ x_3 \\ x_4 \end{array} \right).\]
From the first row, we obtain a necessary condition for the existence of integer solutions, namely that $q-m$ divides $x_1 + x_4$.  Similarly, if $sv - tu = -1$, the corresponding calculation tells us that $q -m $ divides $x_1 - x_4$.

We then consider the possible values of $\vect{x} = \as^{h'} \bs^{j'} \es^{k'} [K_2]$. Since $\as$ acts as scaling by $-1$, it suffices to consider $(j', k') \in \set{(0, 1), (1, 0), (1,1)}$.
\[ [K_2] = (q, m-q, m-q, m-q),\]
\[ \bs[K_2] = (q, q-m, m-q, q-m),\]
\[ \es[K_2] = (m-q, m-q, m-q, q),\] and
\[ \bs \es [K_2] = (m-q, q-m, m-q, -q).\]
So $x_1 + x_4 $ and $x_1 - x_4 $ are both in the set $\set{m, \pm(2q -m) }$. Since $(2q-m, q-m) = (q, q-m) = (m, q-m) = (q, m) =1$, the only way that $q-m$ divides $x_1 + x_4$ or $x_1 - x_4$ is if $q-m = 1$.

But in this case, our system, with $sv - tu = 1$, becomes
\[ \left( \begin{array}{cccc} 1 & 0 & 0 & 0 \\ 0 & 0 & 1 & 0 \\ 0 & 0 & m & -m \\ -m & m & 0 & 0 \end{array} \right| \left. \begin{array}{c} x_1 + x_4 \\ x_2 - x_3 \\ x_3 \\ x_4 \end{array} \right).\]
Then, further elementary row operations result in 
\[ \left( \begin{array}{cccc} 1 & 0 & 0 & 0 \\ 0 & 0 & 1 & 0 \\ 0 & 0 & 0 & -m \\ 0 & m & 0 & 0 \end{array} \right| \left. \begin{array}{c} x_1 + x_4 \\ x_2 - x_3 \\ -mx_2 + (1-m)x_3 \\ mx_1 + (m+1)x_4 \end{array} \right),\]
so from the third row we find that a further necessary condition for an integer solution is that $m$ divide $x_3$. The same result is obtained under the assumption that $sv - tu = -1$. However, $x_3 = m-q = -1$. Thus $m = 1$, and $q-m = 1$ implies that $q=2$, while we assumed that $q > 2$. (Note, in fact, that in this case, $K_1 = K_2 = K(1, 2, 1, -1)$ have exactly the same parameters, as twisted torus knots, and represent the unknot.)

Thus, when $k = 0$, there is no orientation-preserving homeomorphism of $S^3$ fixing $F$ whose induced map on $H_1(F)$ takes $[K_1]$ to $[K_2]$, so there is thus no such homeomorphism taking $K_1$ to $K_2$. Any orientation-reversing homeomorphism would simply be the composition of an orientation-preserving homeomorphism with a mirror reflection, which will certainly not carry $K_1$ to $K_2$.
\end{proof}

We remark that this technique, surprisingly, does not demonstrate an obstruction in the case that $k =1$. In fact, there \emph{is} an induced map on homology taking $[K_1] = [K(q + m, q, m, -1)]$ to $[K_2] = [K(2q - m, q, q - m, -1)]$, namely $\as \bs \es \gs \ds^{-1} \gs \ds^2$. It remains unclear, however, whether the associated homeomorphism does actually carry $K_1$ to $K_2$. Future work will investigate these distinctions between the extended Goeritz group and the induced actions on homology.

\section*{Acknowledgments}
We would like to thank the referee for very helpful and detailed comments. 

\bibliographystyle{hplain}
\bibliography{AdditionalFinalarxivVersion.bib}
\end{document}